\documentclass{amsart}

\usepackage{amsmath, amsthm, amssymb, tikz, xypic, appendix}

\newtheorem{thm}{Theorem}[section]

\theoremstyle{definition}
\newtheorem{lem}[thm]{Lemma}
\newtheorem{cor}[thm]{Corollary}
\newtheorem{prop}[thm]{Proposition}

\newtheorem{dfn}[thm]{Definition}
\newtheorem{prb}[thm]{Problem}
\theoremstyle{remark}

\newcommand{\im}{\operatorname{im}}

\newcommand{\R}{\mathbb{R}}
\newcommand{\N}{\mathbb{N}}
\newcommand{\Z}{\mathbb{Z}}

\newcommand{\bb}[1]{\mathbb{{#1}}}
\newcommand{\inv}{^{-1}}
\newcommand{\ip}[1]{\langle {#1} \rangle} 
\newcommand{\sm}{\smallsetminus}
 
\newcommand{\dom}{\operatorname{dom}}

\newcommand{\coind}{\operatorname{coind}}

\newcommand{\aut}{\operatorname{Aut}}

\newcommand{\res}{\upharpoonright}

\newcommand{\kmix}{\mathrm{Mix}}
\newcommand{\kmin}{\mathrm{Min}}
\newcommand{\kfree}{\mathrm{Free}}
\newcommand{\kfin}{\mathrm{Fin}}

\newcommand{\lra}{\Leftrightarrow}

\newcommand{\s}[1]{\mathcal{{#1}}} 

\renewcommand{\bf}{\mathbf}
\newcommand{\ax}{\mathbf{a}}
\newcommand{\bx}{\mathbf{b}}

\newcommand{\cay}{\operatorname{Cay}}
\newcommand{\sch}{\operatorname{Sch}}

\newcommand{\Axn}{\operatorname{Act}}
\newcommand{\homeo}{\operatorname{Homeo}}
\newcommand{\id}{\operatorname{id}}
\newcommand{\colim}{\operatorname{colim}}
\newcommand{\mspec}{\operatorname{mspec}}
\newcommand{\wdth}{\operatorname{width}}

\title{Topological weak containment}
\author{Riley Thornton}
\begin{document}

\maketitle

\section{Introduction}

Measure theoretic weak containment is a preorder on probability measure preserving group actions that has several guises: local-global convergence, factors of ultrapowers, approximate conjugacy, and $\Sigma_1$-equivalence to name a few. It has shown these many faces in such diverse areas as geometric group theory, combinatorics, representation theory, and model theory. More recently, a theory of weak containment for topological actions has come into view. This article is a survey and expansion of these recent developments.

The various equivalent definitions of measure theoretic weak containment mentioned above suggest analogous topological definitions. Elek introduced a topological analogue of local-global convergence inspired by finite combinatorics, and Zucker has begun studying topological ultrapowers. The main contributions of this paper are to introduce the perspectives of approximate conjugacy and model theory and to establish some basic features of the weak containment order on Cantor flows. 

In the setting of zero-dimensional spaces, the four approaches to topological weak containment we consider are all equivalent. We'll give much more detail below, but for shifts of finite type, weak containment refines the relation of constant time local reducibility (i.e.~the existence of a continuous equivariant map) and coarsens containment (i.e.~the existence of a factor map). So, weak containment interpolates natural combinatorial and dynamical questions.

For more general spaces, the characterizations in terms of combinatorics, ultrapowers, and model theory all have equivalent translations, but approximate conjugacy is a strictly finer relation. In this more general setting, central coloring arguments extend only to finite dimensional settings and suggest some connection to Bourgin--Yang type questions. For instance, the Borsuk--Ulam theorem gives an example of free actions whose ultraproduct isn't free.

The second half of this paper focuses on actions of general countable groups on Cantor space. We look at various order theoretic features--top and bottom elements, antichains--in various classes of actions and see what dynamical and group theoretic properties they pick out. Top elements of the weak containment order are related to genericity questions. Bottom elements among free actions were studied by Elek and turn out to be intimately related to local algorithms. And, the size of antichains among free and mixing actions detect geometric information about the acting group.

For a class $\s K$ of Cantor flows, an action in $\s K$ is a top element for the class if and only if its conjugacy class approximates all the elements of $\s K$. So, conjugacy is generically ergodic on any $G_\delta$ class with a top element, and studying top elements can help decide genericity questions.

Elek showed that, for any group $\Gamma$, there is always a bottom element among free Cantor $\Gamma$-actions. Bernshteyn and Seward showed that this bottom element maps to a shift of finite type if and only if there is an efficient local algorithm for producing elements of the shift. We show that for, $\Gamma=\Z^2$, this bottom element cannot be a limit of finite actions. (The analogous problem for $F_2$ in the measure theoretic setting is a major open problem).

Lastly, we look at antichains in the weak containment order. We show that the size of antichains among free actions and among free mixing actions detect information about the number of ends of the group. For instance, if $\Gamma$ is 2-ended, then $\Gamma$ has only one free mixing action up to weak equivalence and if $\Gamma$ has infinitely many ends then it has a size $|\R|$ antichain of free mixing actions.

Given the wide range of applications of measure theoretic weak containment, we expect there is much more to learn from studying topological weak containment. We end with some problems suggesting further directions for research.

\subsection{Background}
Before stating our results in detail, we'll give a high-level overview of some of the relevant history and background. Throughout $\Gamma$ is a countable discrete group.

Motivated by graph limit theory, Elek introduced what he calls qualitative convergence of Cantor actions as a topological analogue of local-global convergence \cite{ElekQW}. For an action $\ax:\Gamma\curvearrowright \s C$, a finite window $W\Subset \Gamma$, and a continuous labelling $f:\s C\rightarrow [k]$ (equivalently a labeled clopen partition), the set of $W$-local patterns patterns of $f$ in $\ax$ is 
\[ p_W(f, \ax)=\bigl\{(\gamma\mapsto f(\gamma\cdot x)): x\in \s C\bigr\}\subseteq [k]^W\] For example, if $\ax: \Z\curvearrowright 2^\N$ is the odometer action, $f(x)=x(0)$ gives the first coordinate, and $W=\{0,1\}$, then $p_W(f,\ax)$ contains two patterns:
\[(0\mapsto 0,1\mapsto 1)\quad(0\mapsto 1,1\mapsto 0).\] We can read off from $p_W(f,\ax)$ that $f$ is a 2-coloring of the Schreier graph $\sch(\ax)$.

An action $\ax$ weakly contains $\bx$  (written $\bx\preccurlyeq\ax$) if the local patterns of any labelling of $\bx$ can be simulated by a labelling of $\ax$:
\[\bx\preccurlyeq \ax:\lra \;(\forall W\Subset\Gamma)(\forall f:\s C\rightarrow [k])(\exists g:\s C\rightarrow [k])\;p_W(f,\bx)=p_W(g,\ax).\] And a sequence of actions local-global converges if the set of possible local patterns converges for every fixed $W$ and $k$. Though this wasn't done in Elek's work, one can natural generalize these definitions to closed covers of arbitrarily topological spaces.

So, (at least for Cantor actions) topological weak containment is concerned with questions like whether the Schreier graph of an action has a continuous $3$-coloring or a continuous perfect matching. Actions which are higher up in the weak containment order are easier to decorate. Let us briefly mention that in the measure theoretic setting, one studies the distribution of the pattern seen at a random point rather than yes-no questions about the presence of patterns (so, measure theoretic weak containment is concerned with questions about approximate $3$-coloring and the density of independent sets). Among other things, Elek proved that there is a bottom element among the free actions in analogy to the Abert--Weiss theorem for measure theoretic weak containment. As far as the author can tell, this is the earliest explicit attempt to find topological analogues of results about measure-theoretic weak containment.

Motivated by questions in topological group theory, Zucker introduced a notion of ultra-co-product\footnote{Zucker follows Bankston in calling this construction an ultra-co-product as it is categorically dual to an ultraproduct construction. But, it is completely analogous to what is called the ultraproduct in the ergodic theory. We will also follow Bankston.} for continuous actions \cite{AndyTopWeak}. In the special case of countable discrete groups, this construction has a much simpler presentation which closely parallels the ultrapower construction in ergodic theory. We define ultra-co-powers of actions on spaces via Gelfand duality. That is, we define the ultraproduct of $C^*$-algebra actions and dualize to get the ultra-co-product of topological actions. 

For a nonprincipal ultrafilter $\s U$, elements of $C(\ax^{\s U})$ are limits (along the ultrafilter) of sequences of functions in $C(X)$, and every bounded sequence of functions has an ultralimit. The various algebraic operations--sum, scalar multiple, product, and norm--and the action of $\Gamma$ on $C(X)$ are defined to commute with ultralimits. This means approximate properties of $\ax$ are reflected in exact properties of $\ax^{\s U}$. For instance, if $\ax: \Z\curvearrowright 2^\Z$ is the Bernoulli shift, then there are $e_i,f\in C(2^\Z)$ so that for all $i$
\[e_i^2=e_i, f^2=f, |e_i|_\infty=|f|_\infty=1,\quad (\forall n\in [-i,i])\;\bigl|(n\cdot e_i)f\bigr|_\infty=0.\] So, in $C(\ax^\s U)$, if $[e]=\lim_{\s U} e_i$ and $[f]=\lim_{\s U} f$, we have 
\[[e]^2=[e], [f]^2=[f],\bigl|[f]\bigr|_\infty=\bigl|[e]\bigr|_\infty=1,\quad (\forall n\in \Z)\;\bigl|n\cdot[e][f]\bigr|_\infty=0.\] These idempotents and their products correspond to clopen sets and their intersections; so $\ax^{\s U}$ is not topologically transitive. The measure theoretic analogue of this construction is to take a point realization of an ultrapower of $L^2(X,\mu)$ rather than of $C(X).$

For actions $\ax, \bx$ of $\Gamma$ (on any compact Hausdorff space), $\bx\preccurlyeq\ax$ if $\ax^{\s U}$ factors onto $\bx$. We will show that this agrees with the definition of weak containment in terms of local patterns (this was also shown independently by Zucker during the drafting of this article). So, topological weak containment is concerned with the structure of ultra-co-product flows and with which algebras can simulated up to small error in $C(\ax)$.

Continuous model theory provides a rich theory for understanding the structure of ultraproducts. (For experts, we simply view the group $\Gamma$ as part of the language of our structures; so this approach breaks down in Zucker's work on general topological groups). Note that Gelfand duality tells us $\ax^{\s U}$ factors onto $\bx$ if and only if $C(\bx)$ has an equivariant embedding into $C(\ax^{\s U}).$ A general theorem of model theory tells us $C(\bx)$ embeds into an ultrapower $C(\ax)^{\s U}$ if and only if the $\Sigma_1$ theory of $C(\bx)$ dominates that of $C(\ax)$. So, topological weak containment concerns the $\Sigma_1$ theories of certain metric structures. The analogous observation for measure preserving actions has already been put to work alongside model theoretic techniques to prove some theorems in ergodic theory \cite{TodorTomas}.

Lastly, in the ergodic theory of actions on a standard probability space, weak containment is equivalent to approximate conjugacy in the weak topology. A wrinkle appears in the topological setting: there is only one standard probability space, but there are many metrizable compacta with very different automorphism groups. If we restrict to actions on Cantor space, then approximate conjugacy in the uniform topology \textit{is} equivalent to weak containment. There is a great deal of recent work on conjugacy of Cantor actions and on genericity questions in the space of Cantor actions. The theory of weak containment casts an interesting light on this work.

There are other characterizations of measure theoretic weak containment which one could take inspiration from \cite{BurtonKechrisSurvey}. But, we will content ourselves with the four approaches sketched above. As one might expect from this abundance of characterizations, measure theoretic weak containment appears naturally in a wide variety of contexts \cite{PercolationWeakContainment, TodorTomas, HypergraphWeakContainment, GlobalAspects}. Ideally, we would like to have a comparable list of applications of topological weak containment. Elek, Bernshteyn, and Seward have connected questions about topological weak containment with questions about local algorithms and distributed computing \cite{ContinuousCombo}. And, in this article we expand on that work and connect topological weak containment to questions about generic actions and to ideas from geometric group theory. But, there are great many more avenues to explore.

\subsection{Statement of results}

The simplest definition of weak containment is in terms of local patterns. Fix a countable group $\Gamma$ and a compact Hausdorff space $X$. In general, we can't expect to have clopen partitions of $X$, so we look at local patterns of closed covers. Write $A\Subset B$ to mean $A$ is a finite subset of $B$.
\begin{dfn}
    For $\ax: \Gamma\curvearrowright X$ an action, $\vec C=(C_1,...,C_k)$ a closed cover of $X$, and $W\Subset \Gamma$, the \textbf{$W$-local patterns} of $\vec C$ is the set
    \[p_W(\vec C,\ax)=\bigl\{\sigma\in \s P([k]\times W): \bigcap_{(i,\gamma)\in \sigma} \gamma\cdot C_i\not=\emptyset\bigr\}.\] And, the set of $(W,k)$-local patterns of $\ax$ is
    \[\s P_{W, k}(\ax):=\bigl\{p_W(\vec C,\ax): \vec C\mbox{ a closed cover of }X, |\vec C|=k\bigr\}.\]
\end{dfn}

 The ensemble of local patterns (as $W,k$ varies) determines weak containment. The larger the set of local patterns is for a flow, the easier it is to label the action, and the higher the action is in the weak-containment order.

\begin{dfn}
    For any $\Gamma$-flows $\ax, \bx$, $\ax$ \textbf{weakly contains} $\bx$ (or $\bx\preccurlyeq\ax$) if the local patterns of $\bx$ covers can be simulated by $\ax$ covers, i.e.~
    \[\bx\preccurlyeq\ax\;\lra\;(\forall W\Subset \Gamma,k\in \N ) \; \s P_{W,k}(\bx)\subseteq \s P_{W,k}(\ax).\] If $\ax\preccurlyeq\bx$ and $\bx\preccurlyeq\ax$ we say $\ax$ and $\bx$ are \textbf{weakly equivalent} or $\ax\approx \bx$.
\end{dfn}

\begin{prop}
    The relation $\preccurlyeq$ is a quasi-order on actions of $\Gamma$.
\end{prop}

As suggested in the overview, the ultra-co-power of an action $\ax$ with respect to an ultrafilter $\s U$, denoted $\ax^{\s U}$, is a compactification of $\ax$ with respect to the logic of $C(X)$. The proposition below records the fundamental properties of ultra-co-products:

\begin{prop}[c.f.~{\cite[Chapter 5]{Survey}}]
    Fix an index set $I$, and for each $i\in I$ fix a $\Gamma$-flow $\ax_i: \Gamma\curvearrowright X_i$. Let $X=\colim_{\s U} X_i$ be an ultra-co-product of the $X_i$'s, and let $\ax=\colim_{\s U}\ax_i$ be the ultra-co-product action. There is a partial function $\lim_{\s U}: \prod_{i\in I} C(X_i)\rightharpoonup C(X)$ (called the ultralimit) so that 
    \begin{enumerate}
        \item Any bounded sequence $\ip{f(i):i\in\omega}$ with $f(i)\in C(X_i)$ has an ultralimit $C(X)$, and any $f\in C(X)$ is the ultralimit of some bounded sequence
        \item All algebraic operations in $C(X)$ commute with ultralimits, e.g.:
        \[3g+h(\gamma\cdot_{\ax} f)=\lim_{\s U} 3g_i+h_i(\gamma\cdot_{\ax_i} f_i).\]

        \item Norms commute with ultralimits: $|\lim_{\s U} f_i|_\infty=\lim_{\s U} |f_i|_{\infty}$, where the latter limit is taken in $\R$.
    \end{enumerate}
\end{prop}

(This proposition is essentially the definition of ultraproduct. We mention it here to ground the following discussion). Various natural properties pass to ultralimits and ultra-co-products:

\begin{thm}
    Suppose for each $i$, $\ax_i: \Gamma\curvearrowright X_i$ is a free flow and the covering dimension of $X_i$ is at most $n$. Then, $\colim_{\s U} \ax_i$ is also free, and $\dim(\lim_{\s U} X_i)\leq n$.
\end{thm}

But, the dimension assumption is necessary.

\begin{thm}
    There is a sequence of free actions $\ax_i: \Z\curvearrowright X_i$ so that $\colim_{\s U} \ax_i$ is not free.
\end{thm}

Roughly, the issue is that properties involving only finitely many functions pass to the limit, but that for general spaces there is no uniform bound on the number of functions we need to witness free-ness. Similarly, properties that require quantification over arbitrary group elements are not likely to pass to the limit. For instance,
\begin{prop}
    If $\Gamma$ is infinite and $\ax_i$ is free for each $i$, then $\colim_{\s U} \ax_i$ is not topologically transitive.
\end{prop}

Ultraproducts (with respect to appropriate ultrafilters) are saturated in the sense that any countable set of equations has a solution in an ultraproduct if and only if each finite subset of the equations has a solution. The next proposition gives an instance of this phenomenon. We write $\ax ^{\s U}$ for the ultra-co-product of the constant sequence at $\ax$. 

\begin{prop} \label{prop:evocative}
    Suppose $X$ is separable, $\s U$ is a nonprincipal ultrafilter on $\N$, $\bx:\Gamma\curvearrowright X$ is a $\Gamma$-flow, and for any finite list of open sets $U_1,...,U_n\subseteq X$ there is a map $f\in \hom(\ax, \bx)$ whose image meets each $U_i$. Then, there is a factor map in $\hom(\ax^{\s U},\bx).$
\end{prop}

So we can break the problem of building a factor map down into the problem of building maps which meet any finite list of opens. If $\bx$ is a shift of finite type, the converse of the above proposition is also true.

We can define weak containment using the ultra-co-product.

\begin{thm}
For any $\Gamma$-flows, $\ax$ {weakly contains} $\bx$ iff $\ax^{\s U}$ factors onto $\bx$.
\end{thm}

Continuous model theory provides (among other things) a rich language of $\R$-valued functions on $C^*$-algebras which commute with ultralimits. The fragment of this language with only existential quantifiers is denoted $\Sigma_1$. A general result of model theory characterizes weak containment in terms of $\Sigma_1$.

\begin{thm}
   For any $\ax,\bx$, $\bx\preccurlyeq \ax$ if and only, for all $\phi\in \Sigma_1$, $\phi^\ax\leq \phi^\bx.$
\end{thm}

It follows (either from this theorem or the definition of weak containment in terms of local patterns) that there is only set of weak equivalence classes, and not a proper class. The L\"owenheim--Skolem theorem refines this observation. Any weak equivalence class can be realized on a metrizable space:

\begin{thm}
    For any $\Gamma$-flow $\ax$, there is a flow $\ax'$ on a compact metric space so that $\ax\approx \ax'.$
\end{thm}

We endow the space of actions with the uniform topology.

\begin{dfn}
    For a compact metric space $X$, $\homeo(X)$ is the group of homeomorphisms of $X$ with the topology of uniform convergence. The space of actions of $\Gamma$ on $X$ is 
    \[\Axn(\Gamma, X):=\{\ax\in \homeo(X)^{\Gamma}: (\forall \gamma,\delta\in \Gamma)\;\ax(\gamma\delta)=\ax(\gamma)\circ \ax(\delta).\}\] 
\end{dfn} 

Note that $\homeo(X)$ is a Polish space with metric $\sup_{x\in X} d_X(f(x),g(x)),$ and $\Axn(\Gamma, X)$ is a closed subset of $\homeo(X)^{\Gamma}$ (so also a Polish space).

The action of $\homeo(X)$ on $\Axn(\Gamma, X)$ by conjugation has been intensively studied in the past few decades.

\begin{dfn}
    For $\rho\in \homeo(X)$ and $\ax\in \Axn(\Gamma, X)$, $(\rho\cdot \ax)(\gamma)=\rho\inv\circ \ax(\gamma)\circ \rho$.
\end{dfn}

The conjugacy action is usually too complicated to consider in generality. (Its orbit equivalence relation is complete analytic even for $\Gamma=\Z$ and $X=\s C$ \cite{GaoCamerlo}). There is a general process--the topological Scott analysis--for approximating an action of a Polish group with more tractable data \cite[Chapter 6]{hjorth}. The coarsest data we can attach to a point is the closure of its orbit. In the case of $\s C$-actions, this gives yet another characterization of weak containment:

\begin{thm}
    For $\ax, \bx\in \Axn(\Gamma,\s C)$, $\bx\in \overline{\homeo(\s C)\cdot \ax}$ iff $\bx\preccurlyeq\ax$.
\end{thm}

For spaces with non-trivial connected components, approximate conjugacy can be strictly stronger than weak containment. See Proposition \ref{prop:badfactor} for an example. 

We also explain how to extend some of these definitions and results to non-compact flows. 

After setting up all of these definitions, the rest of this paper investigates the shape of the weak-containment order for actions of general countable groups on zero-dimensional spaces. We focus on a handful of classes: free actions, mixing actions, minimal actions, and ultra-co-products of finite actions.

It is straightforward to check that all of these classes have top elements. In $\s C$, top elements correspond to dense conjugacy classes, so we have the following:
\begin{prop}
    If $\s K\subseteq \Axn(\Gamma, \s C)$ is a $G_\delta$ class of actions which is closed under conjugacy, then $\s K$ has a top element with respect to weak containment if and only if the conjugacy action on $\s K$ is generically ergodic. 
\end{prop}
\begin{cor}[Folklore]
    Conjugacy is generically ergodic on all Cantor actions of $\Gamma$, on free Cantor actions, and on minimal Cantor actions.
\end{cor}

Some properties of generic actions follow easily. The following proposition is not new (except in observing many of the same proofs apply to minimal actions), but the theory of weak containment gives a flexible framework for the proofs.

\begin{prop}
    A generic Cantor action or minimal Cantor action is free. If $\Gamma$ is finitely generated, a generic Cantor action is not topologically transitive. And, if $\Gamma$ has finite asymptotic dimension, then a generic action or minimal Cantor action is hyperfinite.
\end{prop}

Free-ness for generic actions seems to be folklore. Hyperfiniteness for generic actions is due to Iyer and Shinko \cite{IyerShinko}.

Elek showed that there is also always a bottom free element \cite{ElekQW}. We improve this slightly:

\begin{thm}
    There is a free action $\bx_\bot\in \Axn(\Gamma, \s C)$ so that, for any free action $\ax$, $\ax\times \bx_\bot\approx \ax$. 
\end{thm}

Bernshteyn and Seward (independently) showed that $\bx_\bot$ admits a local pattern if and only if there is an efficient deterministic local algorithm for producing the pattern. We show that, for $\Gamma=\Z^2$, there is no bottom element among ultra-co-products of finite actions

\begin{thm}
    For any sequence of actions $\ax_i$ of actions of $\Z^2$ on finite sets with $\ax:=\colim_{\s U} \ax_i$ free, there is a sequence of actions on finite sets, $\bx_i$ so that $\bx:=\colim_{\s U} \bx_i$ is free and $\ax\not\preccurlyeq\bx$.
\end{thm}

This means that, for any sequence of finite actions $\ip{\ax_i:i\in \omega}$, there is some pattern that can decorate almost every $\ax_i$ but that cannot be produced by an efficient local algorithm.

Lastly, the width of the weak containment order on $\s C$-flow records geometric information about the group. Suppose $\Gamma$ is finitely generated. If $\Gamma$ is finite, then all free actions of $\Gamma$ on $\s C$ are weakly equivalent. If $\Gamma$ is not a Burnside group or finite, then $\Gamma$ has a perfect antichain of free actions. Similarly, if $\Gamma$ is $0$- or $2$-ended then all free mixing actions of $\Gamma$ are weakly equivalent. But, if $\Gamma$ contains a natural obstruction to $2$-endedness, then $\Gamma$ has a perfect antichain of free mixing actions. 

\begin{dfn}
    For a group $\Gamma$, define 
    \[\kfree(\Gamma)=\{\ax\in \Axn(\Gamma, \s C): \ax\mbox{ is free}\}\]
    \[\kmix(\Gamma)=\{\ax\in \Axn(\Gamma, \s C): \ax\mbox{ is mixing}\}.\]
\end{dfn}

\begin{thm}
    If $\Gamma$ has $\Z$ as a subgroup, then $\wdth\bigl(\kfree(\Gamma)\bigr)=|\R|$. If $\Gamma$ has $\Z^2$ or $F_2$ as a subgroup, then $\wdth\bigl(\kmix(\Gamma)\cap\kfree(\Gamma)\bigr)=|\R|.$
\end{thm}
\begin{cor}
    If $\Gamma$ has infinitely many ends, then $\wdth\bigl(\kmix(\Gamma)\cap\kfree(\Gamma)\bigr)=|\R|.$
\end{cor}

\begin{prop}
    If $\Gamma$ is $2$-ended, then $\wdth\bigl(\kfree(\Gamma)\cap\kmix(\Gamma)\bigr)=1$.
\end{prop}

\subsection{Notation}

Throughout, $\Gamma$ will denote a countable discrete group. We denote the identity element of $\Gamma$ by $\id_\Gamma$, or $\id$ when $\Gamma$ is understood. Formally, we view an action $\ax:\Gamma\curvearrowright X$ as a group homomorphism $\ax: \Gamma\rightarrow S_X.$ So, $\ax(\gamma):X\rightarrow X$ is a bijection (typically a homeomorphism). We write $\gamma\cdot_\ax x$ for $\ax(\gamma)(x),$ and we drop the subscript when the action is understood.


When $\ax:\Gamma\curvearrowright X$ is any action and $E\subseteq \Gamma$, we can form the associated Schreier graph:

\[\sch(\ax, E):=\bigl(X, \{\{x,\gamma\cdot_{\ax} x\}: x\in X, \gamma\in E\}
\bigr).\] In the special case of the right-multiplication action $\rho:\Gamma\curvearrowright\Gamma$, we call the Schreier graph a Cayley graph:
\[\cay(\Gamma, E)=\sch(\rho, E)=\bigl(\Gamma, \{\{\delta,\delta\gamma\}: \gamma\in E, \delta\in \Gamma\}\bigr).\] In particular, $\Gamma$ acts by automorphisms on $\cay(\Gamma, E)$ by left-multiplication.  

For any action $\ax:\Gamma\curvearrowright X$ and any set $A$, we say $\Gamma$ acts by shifting indices on $A^X$ to mean the action
\[\mathbf{s}: \Gamma\curvearrowright A^X\]
\[(\gamma\cdot f)(x)=f(\gamma\inv\cdot x).\] We implicitly view $\Gamma$ as acting by automorphisms on its Cayley graph, so the shift action of $\Gamma$ on $A^\Gamma$ is given by  
\[(\gamma\cdot x)(\delta)=x(\gamma\inv\delta).\]
For any action $\ax:\Gamma\curvearrowright X$ and sets $A\subseteq \Gamma$ and $B\subseteq X$, we write
\[A\cdot_\ax B:=\{\gamma\cdot_\ax x: \gamma\in A, x\in B\}.\] Again, we drop the subscript when $\ax$ is understood.

An action $\ax:\Gamma\curvearrowright X$ is free if no $\gamma\in \Gamma\setminus\{\id\}$ fixes any point. So, if $\ax$ is free, then every component of $\sch (\ax, E)$ is isomorphic to $\cay(\Gamma, E).$

We will mostly be concerned with actions on topological spaces by homeomorphisms. In this case, we call $\ax: \Gamma\curvearrowright X$ a continuous action or topological action. (For actions by a more general topological group $G$, a continuous action should be continuous as a function from $G$ to $\homeo(X)$. But since $\Gamma$ is countable and discrete, this definition is equivalent). If $X$ is a topological space and we write $\ax:\Gamma\curvearrowright X$, we mean a continuous action unless otherwise stated. 

By a $\Gamma$-flow, we mean a continuous action of $\Gamma$ on a compact Hausdorff space. When $A$ is any compact space (such as a finite set), the full $A$-shift is $\mathbf{s}: \Gamma\curvearrowright A^\Gamma$ where $A^\Gamma$ is given the product topology.

Given two continuous actions of $\Gamma$, $\ax:\Gamma\curvearrowright X$ and $\bx:\Gamma\curvearrowright Y$, we write $\hom(\ax, \bx)$ for the set of continuous equivariant maps from $\ax$ to $\bx$:
\[\hom(\ax,\bx)=\bigl\{f\in Y^X: f\mbox{ is continuous, and }(\forall\gamma\in \Gamma, x\in X)\;f(\gamma\cdot_{\ax} x)=\gamma\cdot_{\bx} f(x)\bigr\}. \]

By a subshift of $A^\Gamma$ we mean a closed and shift invariant subset of the full shift. Note that we don't assume $A$ is finite. If $A$ is finite, and $X\subseteq A^\Gamma$ is a subshift, we call $X$ a symbolic subshift. If $X$ is a subshift, we identify $X$ with $\mathbf{s}\res X: \Gamma\curvearrowright X$. So, if $X,Y$ are subshifts, $\hom(X, Y)$ is the set of continuous maps from $X$ to $Y$ that commute with shifting indices. The Curtis--Hedlund--Lyndon theorem characterizes these maps for symbolic subshifts:

For a set $A$, we write $B\Subset A$ to mean that $B$ is a finite subset of $A$.

\begin{prop}[Curtis--Hedlund--Lyndon]
    For subshifts $X, Y\subseteq A^\Gamma$, for any $f\in \hom(X, Y)$, there is a finite window $W\Subset\Gamma$ so that $f(x)(\gamma)$ only depends on ${x\res (\gamma W)}$. In particular, there is some local rule $f_0: A^W\rightarrow A$ so that \[f(x)(\gamma)=f_0\bigl((\gamma\inv\cdot x)\res W\bigr).\]
\end{prop} 

A shift of finite type is a symbolic subshift $X\subseteq A^\Gamma$ where there is a finite window $W\Subset\Gamma$ (called the window of definition of $X$) and list of allowed patterns $P\subseteq A^W$ so that
\[x\in X\;\lra\; (\forall \gamma) \;(\gamma\cdot x)\res W \in P.\] Of course, we could just as easily define $X$ in terms of the forbidden patterns $A^W\setminus P$.

For two actions $\ax:\Gamma\curvearrowright X$ and $\bx:\Gamma\curvearrowright Y$, their product is $\ax\times \bx:\Gamma\curvearrowright X\times Y$
\[\gamma\cdot_{\ax\times \bx} (x,y)=(\gamma\cdot_{\ax}x,\gamma\cdot_{\bx} y).\] The main point of this construction is that mapping to a product is the same as mapping to both $\ax$ and $\bx$: for any $\mathbf c$,
\[\hom(\mathbf c,\ax\times \bx)\cong\hom(\mathbf c, \ax)\times\hom(\mathbf c,\bx).\]

For $X$ a compact Hausdorff space, $C(X)$ is the set of continuous functions from $X$ to $\R$. We view $X$ as a real commutative $C^*$-algebra, i.e.~as an $\R$-vector space equipped with pointwise multiplication and the infinity norm:
\[|f|_\infty=\sup_{x\in X} f(x).\] If $X$ is not compact, then we may still consider $C_b(X)$, the space of bounded continuous functions, as a $C^*$-algebra. Of course, if $\ax:\Gamma\curvearrowright X$ is a continuous action, then $\Gamma$ acts on $C(X)$ by shifting indices: $(\gamma\cdot f)(x)=f(\gamma\inv\cdot x).$ And, so $\Gamma$ acts on the spectrum of $C(X)$ by shifting indices after identifying $\s I$ with its characteristic function:
\[f\in \gamma\cdot \s I\;:\lra \;\gamma\inv\cdot f\in \s I\] so $\gamma\cdot \s I=\{\gamma\cdot f: f\in \s I\}.$

We write $\bf 1$ and $\bf 0$ for the multiplicative and additive identities respectively in an algebra. So, in $C(X)$, $\bf 1$ is the constant $1$ function, and $\bf 0$ is the constant $0$ function.

\subsection{Acknowledgments}
Thank you to Clinton Conley, Sumun Iyer, and Andrew Zucker for helpful conversations. This work was partially supported by NSF MSPRF grant DMS-2202827.

\section{Definitions}

To begin with, let's fill in the technical details of the various definitions sketched in the introduction.  Again, throughout $\Gamma$ is a countable discrete group. While we are primarily interested in flows on compact spaces, we'll give some definitions in greater generality.

\subsection{Local patterns}

We'll start with local patterns in the setting of zero-dimensional spaces. Symbolic subshifts are a useful case to keep in mind. We can think of a subshift, $X\subseteq A^\Gamma$, as a labelling problem-- we want to label vertices of $\Gamma$ by elements of $A$ while avoiding some forbidden patterns. Problems of this kind include coloring Cayley graphs, finding perfect matchings, and building acyclic subgraphs.

The Curtis--Hedlund--Lyndon theorem says that a map $f\in \hom(Y,X)$ is a constant time algorithm that we can run independently at each vertex of $\Gamma$ to convert a solution to the $Y$-labelling problem into a solution to the $X$-labelling problem.

\begin{dfn}[Elek \cite{ElekQW}]
    For a continuous action $\ax:\Gamma\curvearrowright X$ with $X$ a zero-dimensional space (not necessarily compact), $W\Subset \Gamma$, a finite alphabet $A$, and $f: X\rightarrow A$ continuous, the set of \textbf{$W$-local patterns of $f$} in $\ax$ is the set
    \[p_W(f,\ax):=\{p\in A^W: (\exists x\in X)(\forall \gamma\in W)\; f(\gamma\cdot_{\ax} x)=p(\gamma)\}.\]

    The \textbf{$(A,W)$-local patterns} of a zero-dimensional flow $\ax$ is the set \begin{align*}\s P_{A,W}(\ax):= &\, \{p_W(f,\ax): f\in A^X\mbox{ is continuous}\} \\ = &\, \bigl\{\{f(x)\res W: x\in X\}: f\in \hom(\ax, A^\Gamma)\bigr\} .\end{align*} 
\end{dfn}

Among other things, the $(A,W)$-local patterns of a flow $\ax$ tell us if $\hom(\ax, X)=\emptyset$ or not for any shift of finite type $X$ (with alphabet $A$ and window of definition $W$).

Local patterns also tell us what patterns we can make appear in the image of a map. As an example, consider the full shift $A^\Gamma$. Every flow can map into $A^\Gamma$, since it has a fixed point. But, it is a much subtler question to ask which flows have the same local patterns as $A^\Gamma$. 

We could generalize beyond zero-dimensional spaces by linearizing-- replace clopen sets with $[0,1]$-valued continuous functions and intersections with pointwise products. We do essentially this in Section \ref{sec: ultraproducts}. A more combinatorial generalization is to consider closed covers rather than clopen partitions.

\begin{dfn}
    For an action $\ax:\Gamma\curvearrowright X$, a closed cover $\vec C=(C_1,..., C_n)$ of $X$, and a finite window $W\Subset \Gamma$
    \[p_W(\vec C,\ax)=\left\{\sigma\subseteq [n]\times W: \bigcap_{(i,\gamma)\in \sigma} \gamma\inv\cdot A_i\not=\emptyset\right\}.\]
    And $\s P_{W,n}(\ax)=\{p_W(\vec A,\ax): |\vec A|=n\}.$
\end{dfn}

It's an exercise in point-set topology to check that if $X$ is zero-dimensional and normal then for any closed cover $\vec A$ there is a clopen cover $\vec C$ so that $p_W(\vec A,\ax)=p_W(\vec C,\ax)$. Then, $p_W(\vec C,\ax)$ can be recovered from the $W$-local patterns of any function $f$ where the partition $\{f\inv[i]:i\in [n]\}$ generates $\vec C$. Conversely, $p_W(f,\ax)$ can be recovered from $p_W(\vec C, \ax)$, where $\vec C$ contains $f\inv[i]$ and $X\setminus f\inv[i]$ for all $i\in [n]$. So, for zero-dimensional spaces, the families $\s P_{W,n}(\ax)$ and $\s P_{A,W}(\ax)$ determine each other.

Our first definition of weak containment is that $\bx\prec\ax$ if and only if every local pattern of $\bx$ can be simulated by a continuous labelling or closed cover $\ax$:

\begin{dfn}
    For $\ax,\bx$ continuous actions of $\Gamma$ on normal spaces, $\ax$ \textbf{weakly contains} $\bx$, or $\bx\preccurlyeq\ax$, if for all $W\Subset \Gamma$ and $n\in \N$,
    \[\s P_{W,n}(\bx)\subseteq \s P_{W,n}(\ax).\]
    We say $\ax$ and $\bx$ are \textbf{weakly equivalent}, or $\ax\approx\bx$, if $\ax\preccurlyeq\bx$ and $\bx\preccurlyeq\ax$.
\end{dfn}

Observe that $\preccurlyeq$ is a quasi-order since it is comparing the values of some set-valued functions pointwise. So, $\approx$ is an equivalence relation.

When $\ax$ and $\bx$ are flows on zero-dimensional spaces, $\bx\preccurlyeq\ax$ if and only if the local patterns of continuous labellings of $\bx$ can be simulated by labellings of $\ax$.

The following proposition gives the exact relationships between local patterns and shifts of finite type in the zero-dimensional setting. This is essentially a translation of the definition of $\s P_{A,W}(X)$ for $X$ a subshift.

\begin{prop}\label{prop:SFT containment}
    For $X$ a subshift of $A^\Gamma$ with $A$ finite and $\ax$ a flow on $Y$, $\ax$ weakly contains $X$ iff, for any shift of finite type $\tilde X\supset X$ and clopen sets $U_1,..., U_n\subseteq A^\Gamma$ with $U_i\cap X\not=\emptyset$ there is a map $f\in\hom(\ax,\tilde X)$ so that $\im(f)\cap U_i\not=\emptyset$.
\end{prop}

The local patterns of an action encode a great deal of dynamical information. For example

\begin{thm}\label{thm: free closed up}
    If $\bx$ is a free $\Gamma$-flow (i.e.~a free continuous action on a compact Hausdorff space $X$), and $\bx\preccurlyeq\ax$, then $\ax$ is free.
\end{thm}
\begin{proof}
   Fix $\gamma\in \Gamma$. By free-ness (and normality), for any $x$ there is some open neighborhood $U_x$ of $x$ so that $\gamma\cdot_{\bx} \overline U_x\cap \overline U_x=\emptyset$. By compactness, the cover $\{\overline U_x: x\in X\}$ contains a finite subcover $\vec C=(C_1,..., C_n)$.

   Let $W=\{\id,\gamma\}.$ Since $\bx\preccurlyeq\ax$, we have some closed cover $\vec D=(D_1,...,D_n)$ of $\ax$ so that
   \[p_W(\vec C, \bx)=p_W(\vec D, \ax).\] In particular, for each $i$, $\gamma\cdot_{\ax} D_i\cap D_i=\emptyset$. But this means that $\ax$ must be free.
\end{proof}

There are limits to what can be expressed by local patterns, of course. We will see later that properties like topological transitivity and mixing can be ruled out by certain patterns but cannot be enforced by any local patterns (see Proposition \ref{prop: generic not top transitive}.) And, local patterns alone give very little topological information.

\begin{prop}
    Suppose $\ax,\bx$ are trivial actions on $X, Y$ respectively (i.e.~for all $\gamma\in \Gamma$ and $x\in X$, $\gamma\cdot_{\ax} x=x$, and likewise for $\bx$), and $X, Y$ have infinitely many connected components. Then, $\ax\approx \bx.$
\end{prop}
\begin{proof}
    For a closed cover $\vec C$, $p_W(\vec C, \ax)$ just records the nerve of $\vec C$. Under these assumptions, $X$ and $Y$ have all finite simplicial complexes as nerves of closed covers.
\end{proof}

\subsection{Gelfand duality}

 To define ultra-co-products and (and later set up the model theory of actions), we will view $\Gamma$-flows through a functional-analytic duality. 

\begin{dfn}
    For $X$ compact and Hausdorff, let $C(X)$ be the algebra of continuous real-valued functions on $X$ with the $\infty$-norm.

    And for $\ax:\Gamma\curvearrowright X$, let $C(\ax)$ be $C(X)$ equipped with the corresponding action of $\Gamma$ by automorphisms:
    \[(\gamma\cdot f)(x)=f(\gamma\inv\cdot x)\]
\end{dfn} 

 The map $\ax \mapsto C(\ax)$ is the object map of a categorical duality between $\Gamma$-flows and $\Gamma$ actions on commutative $C^*$-algebras, namely Gelfand duality\footnote{Gelfand duality more commonly refers to the theory of commutative complex $C^*$-algebras. Of course $C(X)$ and its complexification are interdefinable, so this difference is not important for our purposes.}. The details of this duality can be found in any good textbook on $C^*$-algebras, for instance \cite{GelfandDuality}. We record the main points below.

 For some readers, it might help to recall one construction of the \v Cech-Stone compactification: $\beta X$ is homeomorphic the prime spectrum of the algebra of continuous bounded functions on $X$.

\begin{dfn} For any algeba $A$, define
\[\mspec(A)=\{\s I\subseteq A: \s I\mbox{ is a maximal proper ideal}\}.\]
    For $x\in X$, define \[\s I_x=\{f: f(x)=0\}\in \mspec(C(X)).\] For $\s A$ a commutative $C^*$-algebra and $f\in \s A$, define $\Phi_f\in C(\mspec(\s A))$ by \[\Phi_f(\s I)=r\;:\lra\; f-r\bf 1\in \s I.\]
    For $f: X\rightarrow Y$ continuous, define $C(f): C(Y)\rightarrow C(X)$ by $C(f)(g)=g\circ f.$
\end{dfn}

\begin{prop}
    For any compact Hausdorff $X$, $(x\mapsto \s I_x)$ is a homeomorphism between $X$ and $\mspec (C(X)).$ For any $\s A$, $(f\mapsto \Phi_f)$ is an isomorphism between $\s A$ and $C(\mspec(\s A))$. For any $X,Y$, $C: \hom(X,Y)\rightarrow \hom(C(Y),C(X))$ is a bijection.
\end{prop}

And, the duality extends to dynamics

\begin{prop}
    For any compact Hausdorff $X$, \[\homeo(X)\cong \aut(C(X)).\] For $\Gamma$-flows $\ax: \Gamma\curvearrowright X$, $\bx: \Gamma\curvearrowright Y$, and for $f:X\rightarrow Y$ continuous, $f$ is equivariant if and only if $C(f)$ is equivariant.
\end{prop}

In particular, factor maps are dual to equivariant embeddings.

\begin{prop}\label{prop:factor duality}
   For $\Gamma$ flows, $\ax$ factors on to $\bx$ if and only if $C(\bx)$ admits an equivariant embedding into $C(\ax)$.
\end{prop}

In practice, this all boils down to the fact we can translate any topological statement about $\ax$ into an algebraic or analytic statement about $C(\ax)$. For example:

\begin{lem}\label{lem: dimension lemma}
    For a normal space $X$, $\dim(X)\leq n$ iff any map from $X$ to $\R^{n+1}$ can be uniformly approximated by maps avoiding the origin.
\end{lem}
This is likely known, but we'll include a sketch of a proof.

\begin{proof}
    Suppose that $\dim(X)\leq n$.  According to \cite[3.2.10]{Engelking}, for any closed $A\subseteq X$ and continuous $f: A\rightarrow S^n$, there is a continuous extension $\tilde f: X\rightarrow S^n$. To approximate a function $g:X\rightarrow \R^{n+1}$, we can set $A=\{x: |g(x)|\geq\epsilon\}$ and consider $f:A\rightarrow S^n$ given by $g(x)/|g(x)|.$ Then if $\tilde f$ is an extension, we can set $\tilde g(x)=g(x)$ if $x\in A$ and $\epsilon \tilde f(x)$ otherwise.

    For the converse, given disjoint pairs of closed sets $(A_1,B_1),..., (A_{n+1},B_{n+1})$ as in \cite[3.2.6]{Engelking}, it suffices to find $L_i$ a partition between $A_i$ and $B_i$ (i.e.~$X\setminus L_i= A\cup B$ with $A, B$ relatively open and $A_i\subseteq A, B_i\subseteq B)$ with $\bigcap_{i=1}^{n+1} L_i=\emptyset$.
    
    Consider $g=(g_1,...,g_{n+1}):X\rightarrow [-1,1]^{n+1}$ with $A_i=g_i\inv[-1]$ and $B_i=g_i\inv[1]$. Then we can build the partition by taking $f$ with $|g-f|_\infty<1/1000$ so that $f$ avoids the origin and taking $L_i=f_i\inv[0]$.
\end{proof}

And as a dynamical example:

\begin{lem} \label{lem: finite dimensional coloring}
    If $X$ is compact and Hausdorff, $\rho\in \homeo(X)$ has no fixed points, and $X$ has covering dimension $n$, then there is a closed cover \[X=E_1\,\cup ...\cup\, E_{2n+3}\] so that for each $i$, $\rho[E_i]\cap E_i=\emptyset.$
\end{lem}

Note that this lemma is a generalization of the familiar greedy bound for coloring zero-dimensional spaces.\footnote{One of the first instances of the zero-dimensional coloring argument is in a paper by Krawczyk and Steprans on free-ness of Cech-Stone compactifications. They actually mention higher-dimensional generalizations in their paper; but, their reference is the PhD thesis of Kim, which is the author has been unable to access. \cite{KrawczykSteprans}}
\begin{proof}
    Call a set $S\subseteq X$ independent if $S\cap \rho[S]=\emptyset$. Note that if $S$ is independent then so is any subset of $S$. By compactness and normality, there is finite closed cover of $X$ by independent sets, say $C_1,..., C_k.$ 

    We will show by induction on $k$, that if there is a cover by $k$ independent closed sets, then there is a cover by at most $(2n+3)$-independent sets. The base case is trivial. 
    
    Suppose that any independent cover of $X$ by $k$ sets can be reduced to an independent cover by $2n+3$ sets. And, suppose that $X$ has and independent cover $C_1,...,C_{k+1}$. We just have to get a cover with one set fewer.
    
    Again appealing to normality, can find an independent open sets $U_i\supseteq C_i$ for each $i$. By the definition of dimension, there is an open cover $X=V_1\cup... \cup V_{k+1}$ with $V_i\subseteq U_i$ so that $\bigcap_{i\in A} V_i=\emptyset$ for $A\in \binom{k}{n+2}$.  And, with one last appeal to normality, we can find a closed cover $D'_1,..., D'_{k+1}$ with $D'_i\subseteq V_i$.
    
    Finally, define $D_i$ for $i\leq k$ by \[D_i:=D'_i\cup\{ x\in D'_{k+1}: \rho(x), \rho\inv(x)\not\in V_i\}.\] Clearly each $D_i$ is closed. Because no point is in more than $n+1$ of the $V_i$, for any $x\in D'_{k+1}$ there is some $j\leq 2n+3$ so that $x\in D_j$. So, $\bigcup_{i} D_i=X$. Let's check independence: if $x\in D_i$ then either $x\in D'_i$ or $x\in D'_{k+1}$. In the former case, $\rho(x)\not\in D'_i$ by independence and $\rho(x)\not\in D_i\sm D'_i$ since $x\in V_i.$ In the latter case, $\rho(x)\not\in D'_{k+1}$ by independence and $\rho(x)\not\in V_i$ so $\rho(x)\not\in D_i.$
\end{proof}
\begin{cor}
    If $\dim(X)\leq n$, then $\rho\in \homeo(X)$ is fixed-point free iff there are $f_1,...,f_{2n+3}\in C(X)$ so that $\max(f_1,...,f_{2n+3})(x)\geq 1$ for all $x$, $f_i$ is $[0,1]$-valued, and $(f_i\circ \rho)f_i=0$.
\end{cor}

Using this kind of translation, we will apply the well-developed model theory of $C^*$-algebras to study actions.

\subsection{Ultra(co)products} \label{sec: ultraproducts}

Ultra-co-products are a way to compactify flows with respect to the logic of $C(\ax)$. By this we mean that a list of equations in $C(\colim_{\s U} \ax_i)$ has a solution iff there is a sequence of approximate solutions in the $C(\ax_i)$'s. 

The definition of the ultraproduct of a family of algebras is standard. See the survey of Ben Yaacov, Berenstein, Ward Henson, and Usvyatov for the general definition \cite{Survey}. The ultra-co-product of a sequence of flows is the dual action of the ultraproduct of their algebras.

 Consider any index set $I$. Let $\ip{\ax_i: i\in I}$ be a sequence of $\Gamma$-flows and let $\s U$ be a nonprincipal ultrafilter on $I$. Recall that, for a sequence $\ip{x_i:i\in I}$ in a metric space,
    \[\lim_{\s U} x_i=y \;:\lra\;(\forall \epsilon>0)\; \{i: d(x_i,y)<\epsilon\}\in \s U.\] Every sequence in a compact metric space has a unique ultralimit.

\begin{dfn}
   For an ultrafilter $\s U$ on an index set $I$, The \textbf{ultraproduct of algebras}, $\lim_{\s U} C(\ax_i)$ is the algebra whose domain is
    \[\{\ip{f_i: i\in I} \in \prod_I C(\ax_i):(\exists N)(\forall i) |f|<N\}/ \sim_{\s U}\]
    \[\ip{f_i: i\in I}\sim_{\s U} \ip{g_i:i\in I} \;:\lra \;\lim_{\s U} |f_i-g_i|=0.\] If we need to disambiguate which index we take a limit along, we write $\lim_{\s i\rightarrow \s U} C(\ax_i).$

    We write $[f(i)]$ or $\lim_{\s U} f(i)$ for the equivalence class of the sequence $\ip{f(i) :i\in I}$. Often, we will leave the index $(i)$ implicit. The algebraic operations are defined to commute with ultralimits: for $\lambda\in \R$, $[f],[g],[h]\in \lim_{\s U} C(\ax_i)$
    \[ (\lambda[f]+[g])[h]:=[(\lambda f+g)h]\] and likewise we define the action on the algebra for $\gamma\in \Gamma$ by
    \[\gamma\cdot [f]=[\gamma\cdot f]\] and the norm by
    \[\bigl|[f]\bigr|_\infty=\lim_{\s U} \bigl|f(i)\bigr|_{\infty}.\] (Note, this last limit is taken in $\R$. It always exists since $\ip{f(i):i\in I}$ is always a bounded sequence).
\end{dfn}

 Most basic properties of $\lim_{\s U} C(\ax_i)$ follow from a little symbol pushing. (The continuous \L o\'s's theorem, Theorem \ref{thm:los} below, gives a general metatheorem in this vein). For example:

\begin{prop}
    For any sequence of $\Gamma$-flows, $\ax_i$, the ultraproduct of algebras is a commutative $C^*$-algebra. Each $\gamma\in \Gamma$ acts on $\lim_{\s U}C(\ax_i)$ by an isomorphism. 
\end{prop}

\begin{proof}
    The commutativity of multiplication goes as follows. Consider any $[f],[g]\in \lim_{\s U}C(\ax_i)$, we have
    \[[f][g]=[fg]=[gf]=[g][f].\] The other $C^*$-algebra axioms, and the fact that $\Gamma$ acts by isomorphisms are similar. 
\end{proof}

The ultra-co-product of a family of flows is a point realization of the ultraproduct of their algebras.

\begin{dfn}
    For an $I$-indexed sequence of $\Gamma$-flows, $\ip{\ax_i: i\in I}$, where $\ax_i: \Gamma\curvearrowright X_i$, and an ultrafilter $\s U$ on $I$, their \textbf{ultra-co-product} is the spectral flow of their ultraproduct of algebras:
    \[\colim_{\s U} X_i:=\mspec(\lim_{\s U} C(X_i))\quad \gamma\cdot_{\colim_{\s U}\ax_i}\s I=\{\gamma\inv\cdot f: f\in \s I\}.\] We also write $\colim_{i\rightarrow \s U} \ax_i$ if we want to be more explicit.

    For $f\in \lim_{\s U} C(\ax_i)$, $\phi(f)\in C(\colim_{i\in \s U} \ax_i)$ is the canonical dual function defined by
    \[\phi(f)(x)=r\; :\lra\; f-r\bf 1\in x\] (recall that $x$ is a maximal ideal here.)

    If $\ax_i=\ax$ for all $i$, we call the ultra-co-product an \textbf{ultra-co-power} and write $\ax^{\s U}=\colim_{\s U} \ax_i.$
\end{dfn}

    We typically conflate $f$ and $\phi(f)$. In particular, for any sequence of functions $\ip{f(i): i \in I}$, we will usually write $[f]$ instead of $\phi([f])$ for the ultralimit as a function on $\colim_{i\in \s U} \ax_i$.

We establish most properties of $\colim_{\s U}\ax_i$ by translating through Gelfand duality and applying \L o\'s-style reasoning as before. For example, our lemmas from the previous section give us:

\begin{prop} \label{prop: covering dimension}
    If $X_i$ has covering dimension at most $n$ for all $n$, then $\colim_{\s U} X_i$ has covering dimension at most $n$.
\end{prop}
\begin{proof}
   Suppose each $X_i$ has covering dimension at most $n-1$ for each $i$. Fix any $\epsilon>0$ and $[f_1],...,[f_n]\in C(\colim_{\s U} X_i)$. By Lemma \ref{lem: dimension lemma}, there are functions $[g_1],..., [g_n]$ so that for each $i\in I$ and $k\in [n]$, \[\bigl|g_k(i)-f_k(i)\bigr|_\infty<\epsilon \quad \quad (\forall x\in X_i) \; \bigl(g_1(i)+...+g_n(i)\bigr)^2(x)>\epsilon^2/4.\] We have to use some tricks to check that this last inequality passes to the limit. Note that $h(x)\geq r$ for all $x$ iff $h-r\bf 1$ is nonnegative iff $h-r\bf 1=w^2$ for some $w$. So, fix $w(i)$ for all $i$ with \[\bigl(g_1(i)+...+g_n(i)\bigr)^2-(\epsilon^2/4)\bf 1=w_k(i)^2.\] We then have,
   \[\bigl|[f_k]-[g_k]\bigr|_\infty=\lim_{\s U}\bigl|f_k(i)-g_k(i)\bigr|_\infty\leq\epsilon\] and \[\bigl([g_1]+...+[g_n]\bigr)^2-(\epsilon^2/4)\bf 1=[w]^2,\] so $\min_x \bigl([g_1]+...+[g_n]\bigr)^2(x)\geq \epsilon^2/4$.

   This means we can approximate any map from $\colim_{\s U} X_i$ to $\bb R^n$ by a map which avoids the origin as desired.
\end{proof}

And, as a dynamical example, we have the following:

\begin{prop}\label{prop:freeness}
    Suppose we have flows $\ax_i:\Gamma\curvearrowright X_i$ and that for all $i\in I$, $X_i$ has covering dimension at most $n$ and $\ax_i$ is free. Then $\ax=\colim_{\s U} \ax_i$ is free.
\end{prop}
\begin{proof}
Fix $\gamma\in \Gamma$. Using Lemma \ref{lem: finite dimensional coloring}, for each $i\in I$ we can find $f_1(i),..., f_{2n+3}(i)\in C(\ax_i)$ so that \[\bigl(\gamma\cdot_{\ax_i}f_k(i)\bigr)f_k(i)=\bf 0 \quad \quad \bf 0\leq f_k(i)\] \[f_1(i)+...+f_{2n+3}(i))\geq \bf1,\]  where the inequalities are all pointwise. 

So, if $w(i)^2=f_1(i)+...+f_{2n+1}(i)-\bf 1$, $v_k(i)^2=f_k(i)$, and $u_k(i)^2=1-f_k(i)$ we have
\[[f_k](\gamma\cdot [f_k])=[f_k(\gamma\cdot f_k)]=[\bf 0]=\bf 0\] and
\[[f_k]=[v_k^2]=[v_k]^2\] so $[f_k]\geq \bf 0$. Similarly, $[f_k]\leq \bf 1$ and $[f_1]+...+[f_k]\geq \bf 1.$ 

Suppose towards contradiction that $\gamma\cdot x=x$. Then, $[f_1]+...+[f_k]\geq 1$, so $[f_i](x)>0$ for some $x$. But then, \[0=\bigl([f_i](\gamma\cdot [f_i])\bigr)(x)=([f_i](x))([f_i](\gamma\inv\cdot x))=[f_i]^2(x)>0.\] This is absurd, so $\gamma$ has no fixed points under $\colim_{\s U} \ax_i.$  \end{proof}

Note that much the same argument gives that any ultra-co-power of a free flow is free. All we needed was a bound on the number of sets in an independent cover of $X_i$, and if the action is fixed we can use the same cover in every index. We can generalize somewhat:

\begin{dfn}
    For an action $\ax:\Gamma\curvearrowright X$ and $\gamma\in \Gamma$, define $\chi(\gamma,\ax)$ to be the least $n$ so that there are $n$ functions witnessing that $\ax(\gamma)$ has no fixed points, i.e.
    \[\chi(\gamma,\ax)=\min\bigl\{n:(\exists f_1,...,f_n\in C(\ax))\: \bf 0\leq f_i, \sum_{i=1}^n f_i\geq \bf 1, \mbox{ and }f_i(\gamma\cdot f_i)=\bf 0\bigr\}.\] If this set is empty, $\chi(\gamma,\ax)=\infty.$
\end{dfn}
Note that $\ax$ is free if and only if, for all $\gamma\not=\id$, $\chi(\gamma,\ax)<\infty.$

\begin{prop}
    For any $I$-indexed sequence of action $\ip{\ax_i: i\in I}$,
    \[\chi(\gamma,\colim_{\s U}\ax_i)=\lim_{\s U}\chi(\gamma,\ax_i).\]
\end{prop}
\begin{proof}
    The inequality $\chi(\gamma,\colim_{\s U}\ax_i)\leq\lim_{\s U}\chi(\gamma,\ax_i)$ is essentially the content of the proof above. If we have witnesses $f_1(i),...,f_n(i)$ to $\chi(\gamma,\ax_i)\leq n$ for a $\s U$-large set of indices $i$, then $[f_1],...,[f_n]$ witness that $\chi(\gamma,\colim_{\s U} \ax_i)\leq n.$

    For the reverse inequality, suppose that $[f_1],...,[f_n]\in C(\colim_{\s U} \ax_i)$ witness that $\chi(\gamma,\colim_{\s U} \ax_i)\leq n$, i.e.
    \[[f_j]\geq \bf 0\quad \quad [f_j](\gamma\cdot [f_j])=0\]
    \[[f_1]+...+[f_n]\geq \bf 1.\] Then, by considering witnesses $[w_j]$ and $[u]$ with $[f_j]=[w_j^2]$ and $[f_1+...+f_n]=[u^2]$, we have that, for any $\epsilon>0$ and a $\s U$-large set of indices $i$,
    \[f_j(i)\geq -\epsilon \bf 1\quad \quad |f_j(i)(\gamma\cdot f_j(i))|<\epsilon\]\[f_1(i)+...+f_n(i)\geq (1-\epsilon)\bf 1.\] Then, multiplying by small bump functions as appropriate, we can find $f'_1(i),...,f_n'(i)$ witnessing that $\chi(\gamma,\ax_i)\leq n.$
    
\end{proof}

For ultra-co-products where the action can vary, the bounds on dimension in Proposition \ref{prop:freeness} are necessary:
\begin{prop}\label{prop:nonfree limit}
   Fix a nonprincipal ultrafilter $\s U$ on $\N$. Let $\ax_i: \Z/2\Z\curvearrowright S_i$ be the antipodal action:
    \[(-1)\cdot_{\ax_i} \bf x=-\bf x.\] Then, $\colim_{\s U} \ax_i$ is not free.
\end{prop}
\begin{proof}
    Let $\gamma$ be the nonidentity element of $\Z/2\Z$. It suffices to show that $\chi(\gamma, \ax_n)> n$. Suppose towards contradiction there are $f_1,..., f_n: S_n\rightarrow \R$ witnessing that $\chi(\gamma,\ax_n)\leq n$. Then, for
    \[f: S_n\rightarrow \R\]
    \[f(x)=(f_1(x),...,f_n(x))\] the Borsuk--Ulam theorem says that there is some $x\in S_n$ with $f(x)=f(\gamma\cdot x).$ But then, for each $i$, $0=f_i(x)f_i(\gamma\cdot x)=f_i(x)^2$, so $(f_1+...+f_n)(x)=0<1.$
\end{proof}

Bourgin--Yang theory--a branch of algebraic topology--is concerned with the dimension of the points of coincidence of maps between simplicial $G$-spaces for (typically compact) groups $G$. The arguments above suggest that dynamical properties of ultra-co-products might be tied up with similar concerns. For instance, a bit more advanced algebraic topology gives the following.

\begin{thm}
    Fix a nonprincipal ultrafilter $\s U$ on $\N$. There is a sequence of free $\Z$-flows, $\ax_i$, so that $\colim_{\s U} \ax_i$ is not free.
\end{thm}
\begin{proof}
It suffices to find a sequence of $\Z$-actions $\ax_i$ with $\chi(n,\ax_i)\in [i-3,i+3]$ for all $n$.

Let $E_n(\Z/p\Z)$ be the $n^{th}$ join of $\Z/p\Z$ (as a discrete space) with itself. Formally, we define $E_n(\Z/p\Z)\subseteq \R^{n+1}$ inductively as follows:
$E_0(\Z/p\Z)=\{1,2,3,...,p\}$

\[E_{n+1}(\Z/p\Z)=\bigl\{t(x,0,0,...,0)+(1-t)(0,y_1,...,y_n): x\in \Z/p\Z,\; \vec y\in E_{n}(\Z/p\Z)\bigr\}.\] We inductively define an action $\bx^p_{n}:\Z\curvearrowright E_n(\Z/p\Z)$ by \[ n\cdot_{\bx^p_0} x=\mbox{ the remainder of }n+x\mbox{ by }p\]\[n\cdot t(x,\vec 0)+(1-t)(0,\vec y)=t(n\cdot x)+(1-t)(0,n\cdot \vec y).\] It's a straightforward induction to check that, for any $k=1,...,p-1$, the homeomorphisms $\bx^p_n(k)$ and $\bx^p_n(1)$ are conjugate. In particular, $\chi(k,\bx_n^p)=\chi(1,\bx^p_n)$. Since deleting $(\Z/p\Z)\times \{\vec 0\}$ from $E_{n}(\Z/p\Z)$ gives a space homotopic to $E_{n-1}(\Z/p\Z)$ (with an equivariant homotopy), we have that
\[\chi(1,\bx^p_{n+1})\leq \chi(1,\bx^p_n)+3.\]

The free part of the action $\ax: \Z/p\Z\curvearrowright \R^{np}$ given by \[k\cdot (\vec x_1,...,\vec x_p)=(\vec x_{1+k},...,\vec x_{p+k})\] (with indices taken mod $p$) has some finite index (in the sense of \cite[Chapter 4]{BorsukUlam}.) So, for all large enough $k$, if we view $\bx^p_{k}$ as a $\Z/p\Z$ action, any map $f\in\hom(\bx^p_k, \ax)$ must have a fixed point of $\ax$ in its image. In particular, $\chi(1,\bx^p_{k})>n$ by the same argument as in the previous proposition. So $\chi(1,\bx^p_{n})$ tends to $\infty$ with $n$. 

Combining these observations, we have that, for any $k\in \N$, there is some $n_{pk}$ so that $\chi(1,\bx^p_{n_{k}}),...,\chi(p-1,\bx^p_{n_k})\in [k-3,k+3].$ Set $\ax^p_{k}=\bx^p_{n_{k}}.$

Then, $\ax_k=\colim_{p\rightarrow \s U} \ax^p_{k}$ has $\chi(n,\ax_k)\in [k-3,k+3]$ for all $k$. Thus, $\ax_k$ is free, but $\ax=\colim_{\s U} \ax_k$ has $\chi^1(\ax)=\infty$, so $\ax$ is not free.

\end{proof}

Ultralimits of functions are defined so that point-wise multiplication (and scaling and so on) commute with ultralimits. It turns out every point-wise operation commutes with ultralimits, as do suprema and infima.

\begin{prop}\label{prop:pointwise}
    Consider any $\rho: \R\rightarrow\R$. Then, for any $[f]\in \lim_{\s U}C( \ax_i)$, 
    \[[\rho\circ f]=\rho\circ[f],\quad \quad \sup_{x\in X} [f](x)=\lim_{\s U} \sup_{x\in X_i} f(x),\] \[ \mbox{ and } \inf_{x\in x}[f](x)=\lim_{\s U} \sup_{x\in X_i} f(x).\]
\end{prop}
\begin{proof}

    Fix $x\in \colim_{\s U} X_i$, and fix $[f]\in C(\colim_{\s U}X_i)$. By the definition of ultraproducts of algebras, there is some $N$ so that $|f(i)|<N$ for all $i$.
    
    Fix any $\epsilon>0$. By Stone--Weierstrass, there is some polynomial $p$ so that ${|p(\zeta)-\rho(\zeta)|<\epsilon}$ for all $\zeta\in [-N,N]$. So, 
    \[\bigl|[\rho\circ f]-p([f])\bigr|_\infty=\bigl|[\rho\circ f]-[p(f)]\bigr|_{\infty}=\bigl|[\rho\circ f-p(f)]\bigr|_\infty<\epsilon.\] So, for our fixed $x$, \[\bigl|[\rho\circ f](x)-\rho([f](x))\bigr|\leq \bigl|[\rho\circ f]-p([f])\bigr|_\infty+\bigl| p([f](x))-\rho([f](x))\bigr|\leq 2\epsilon.\] Sending $\epsilon$ to $0$ shows that $\rho([f](x))=[\rho\circ f](x).$

    And, \[\sup_x [f](x)=\bigl|[f]+N\bf 1\bigr|_\infty-N=\lim_{\s U} \bigl|f(i)+N\bf 1\bigr|_\infty-N=\lim_{\s U} \sup_x f(i)(x).\]
\end{proof}
We'll leave it to the reader to verify that the same proof works for $n$-ary operations.

The ultra-co-product is defined so that we can freely take limits of maps out of $\colim_{\s U} \ax_i$ into compact subsets of $\R$. The finite additivity of $\s U$ lets us diagonalize against any countable list of constraints. The next theorem applies this idea to building factor maps. We can build a factor map $[f]\in \hom(\ax^{\s U} ,\bx)$ as a limit and ensure $[f]$ is onto by working to piecemeal to ensure each $f(i)$ hits enough open sets.

\begin{lem}
    Fix flows $\ax:\Gamma\curvearrowright X$ and $\bx:\Gamma\curvearrowright Y$ with $Y$ a compact metric space, and fix an ultrafilter $\s U$ on $\N$. Suppose that, for any finite list of open sets $U_1,...,U_n\subseteq Y$ there is $f\in \hom(\ax, \bx)$ so that $\im(f)\cap U_i\not=\emptyset$.  Then, $\ax^{\s U}$ factors onto $Y$.
\end{lem}
\begin{proof}
    By standard coding arguments, we can assume that $Y\subseteq [0,1]^{\N\times \Gamma}$ and that $\Gamma$ acts by shifting indices:
    \[(\gamma\cdot x)(i,\delta)=x(i,\gamma\inv\delta).\] (For instance, we can map $x\in X$ to $\ip{d(\gamma\cdot x,x_n):\gamma\in \Gamma, n\in \N}$, where $\{x_i: i\in\N\}$ is some countable dense set.)

    Since $Y$ is a metric space we have a countable basis for the topology, $\{U_i:i\in\N\}$.  For each $i\in \N$ we can find $f(i)\in \hom(\ax, \bx)$ so that $U_{j}\cap \im f(i)\not=\emptyset$ for all $j<i$. Let $f_{n,\gamma}(i)$ be the $(n,\gamma)$-coordinate function of $f(i)$:
    \[f(i)(x)=\ip{f_{n,\gamma}(i)(x): n\in \N, \gamma\in \Gamma}.\]

    We can define $[f]: X\rightarrow [0,1]^{\N\times \Gamma}$ by $([f](x))(n,\gamma)=f_{n,\gamma}(x)$. By equivariance of $f(i),$
    \[\gamma\cdot [f_{n,\delta}]=[\gamma\cdot f_{n,\delta}]=[f_{n,\gamma\inv\delta}]\] so $f$ is equivariant. 

     For $W\Subset \N\times \Gamma$ let $\pi_W:[0,1]^{\N\times \Gamma}\rightarrow [0,1]^W$ be the projection function. Define $\phi_W:[0,1]^W\rightarrow \R$ by
     \[\phi_W(x)=d(x, \pi_W(Y)).\]
     By the previous, Proposition \ref{prop:pointwise}, \[\sup_{x\in X^{\s U}} \phi_W\circ [f]=\lim_{\s U} \sup_{x\in X} \phi_W \circ f=0.\] So, for all $x\in X^{\s U}$, $d([f](x),Y)=0$. Thus $\im ([f])\subseteq Y$. 

     Similarly, let $\psi_{p,W}(x)=d(\pi_W(x), \pi_W(p))$. Since $f(i)$ eventually meets each open set in $Y$, for each $p\in Y$ and $W\Subset \N\times \Gamma$, $\inf_{x\in X} \psi_{p,W}\circ f(i)(x)$ tends to $0$. So, again by the previous proposition, \[\inf_{x\in X^\s U} \psi_{p,W}\circ [f](x)=\lim_{\s U} \inf_{x\in X} \psi_{p,W} \circ f(x)= 0.\] Thus, for $y\in Y$, $d(y,\im([f]))=0$. This means $f$ is onto.
\end{proof}

The above theorem should remind the reader of Proposition \ref{prop:SFT containment}. In fact, we have

\begin{thm}

    For $\Gamma$-flows $\ax$ and $\bx$, $\ax\preccurlyeq\bx$ iff $\bx^{\s U}$ factors onto $\ax$ for some $\s U$.
\end{thm}

We'll delay the proof until Section \ref{sec: equivalences}.

\subsection{Model theory}

We said before that the ultra-co-power of a flow is a compactification of the flow with respect to the logic of $C(X)$. In this section, we'll make more precise what we mean by ``the logic of $C(X)$". Continuous model theory provides a rich language of functions and invariants that commute with ultralimits, and studies the geometry of sets and structures definable in this language. Much of our discussion of $\lim_{\s U}C(\ax_i)$ in the previous section is a specialization of this more general theory. We will highlight the main points here and refer the reader to the survey \cite{Survey} for more details.

First we define the nouns of our language, called terms in model theory. These are variable elements of $C(\ax)$ built up from the usual constants and algebraic operations, such as
\[\gamma\cdot(x_1+3x_2)x_3-\bf 1.\]

\begin{dfn}
    The set of \textbf{terms} in the algebra $C(\ax)$ is the set of functions from powers of $C(\ax)$ to $C(\ax)$ defined recursively as follows:
    \begin{itemize}
        \item The function $\pi_i(x_1,...,x_n)=x_i$ are terms
        \item The constant function $\tau(x_1,...,x_n)=\bf 1$ is a terms
        \item If $\tau_1,\tau_2$ are terms, $\gamma\in \Gamma$, and $\alpha\in \R$, then the following are all terms:
        \[\alpha\tau_1,\quad \tau_1+\tau_2,\quad \tau_1\tau_2, \quad\gamma\cdot \tau_1. \]
    \end{itemize}
\end{dfn}
This definition is technically not quite correct. A term should really be a formal symbol with a canonical interpretation as an operation on any algebra $C(\ax)$. 

\begin{dfn}
    For an $\R$-algebra $C$ equipped with a $\Gamma$ action, $f_1,...,f_n\in C$, and an $n$-place term $\tau(x_1,...,x_n)$, we write
    \[\tau^{C}(f_1,...,f_n)\] for the value of the term applied to $f_1$,..., $f_n$ interpreted in $C$. We abbreviate $\tau^{C(\ax)}$ by $\tau^{\ax}$.
\end{dfn}

For example $\tau(x)=x-\gamma\cdot x$ is a term. If $\gamma$ acts trivially in $C$, $\tau^C(x)=\bf 0$ for all $x$. Otherwise $\tau^C$ depends sensitively on $x$. 

The fact we can interpret terms in different algebras is a fussy point, but it's an important one. For instance, it allows us to state the following proposition (whose proof is a trivial induction on the construction of terms).

\begin{prop}
    If $\tau$ is a term, then for any $[f_1],...,[f_n]\in C:=\lim_{\s U}C(\ax_i)$
    \[\tau^C([f_1],...,[f_n])=\lim_{\s U} \tau^{\ax_i}\bigl(f_1(i),...,f_n(i)\bigr).\]
\end{prop}

Note that $\tau^C$ is always a uniformly continuous function when restricted to any bounded set in $C$. We need to restrict to bounded sets to ensure uniform continuity, which will lead to some technical wrinkles in later definitions.\footnote{Another approach would be to only work with the unit ball in $C(\ax)$. This is the approach taken in the survey \cite{Survey}.}

Note that $|\tau|_{\infty}$ is not a term. It couldn't possibly be--it takes values in $\R$, not $C(X)$. We call definable functions into $\R$ predicates. We think of these as relations in our language and their value in $\R$ represents how far they are from holding. The two-variable predicate $|f-g|_\infty$ says how far $f$ and $g$ are from being equal, the one place predicate $\inf_{g} |f-g^2|_\infty$ says how far $f$ is from being non-negative, etc. The universal and existential quantifiers are $\sup$ and $\inf$ respectively.

\begin{dfn}
    The set of \textbf{predicates} in the algebra $C(\ax)$ is the set of function from powers for $C(\ax)$ to $\R$ defined recursively as follows:
    \begin{enumerate}
        \item For any term $\tau$, $|\tau|_\infty$ is a predicate
        \item If $\phi_1,...,\phi_n$ are predicates and $c: \R^n\rightarrow \R$ is uniformly continuous, then $c(\phi_1,...,\phi_n)$
        \item If $\phi(x, \vec y)$ and $\psi(x,\vec y)$ are predicates then so are $\inf_{| x|< \psi(x,\vec y)} \phi(x,\vec y)$ and $\sup_{| x|<\psi(x, \vec y)} \phi(x,\vec y).$
    \end{enumerate}
   An \textbf{existential predicate} is one of the form
   \[\inf_{x_1,...,x_n} \phi(x_1,...,x_n, y_1,...,y_n)\] where $\phi$ is predicate with no occurrence of $\inf$ or $\sup$.

   A\textbf{ sentence }is a 0-place predicate. An \textbf{existential sentence} is a 0-place existential predicate.
\end{dfn}

Existential predicates are also sometimes called $\Sigma_1$ predicates. The definitions are arranged so that every predicate is a uniformly continuous function. For example, the following are all predicates:
\[d(f,g)=|f-g|_\infty\quad \psi(f)=d(f,f^2)\quad \phi(f)=\inf_{|g|\leq \sqrt{|f|_\infty}}d\bigl(f, g^2\bigr).\]
Roughly speaking, $\psi$ measures how close $f$ is to being a characteristic function, and $\phi$ measures how close $f$ is to being nonnegative.

The wording of $\phi$ is a little procrustean to meet the definition of predicate. Note that
\[\phi(f)=\inf_{|g|\leq \sqrt{|f|_\infty}} d\bigl(f, g^2\bigr)=\inf_{g} d\bigl(f,g^2\bigr). \] Sometimes we will write $\inf_{g} \phi(g)$ as a predicate when the infimum of $\phi$ is always attained on a uniformly bounded set when $\phi$ is restricted to a bounded set. 

The following are not predicates:
\[\inf_{\gamma\in \Gamma} d(\gamma\cdot f, g)\quad \inf_{g} \max(1, |fg-\bf 1|).\] Quantifiers have to range over bounded sets in $C(\ax)$, not over $\Gamma$ or unbounded sets.

Again, we have the fussy point that predicates should formally be sequences of terms concatenated with symbols for norms, continuous connectives, $\sup$s, and $\inf$s. We can canonically interpret predicates in any structure. 

\begin{dfn}
    For $f_1,...,f_n\in C(\ax)$ and $\phi$ and $n$-place predicate,
    we write $\phi^{\ax}(f_1,...,f_n)$ for the value of the predicate these functions interpreted in $C(\ax)$.
\end{dfn}

Existential predicates are important because they are the broadest class of predicates which are monotone with respect to embeddings of algebras (or, dually, with respect to factors of flows). The following proposition is an easy induction on the construction of predicates:

\begin{prop} \label{prop:embedding}
    If $\phi(x_1,...,x_n)$ is an existential predicate, and $F:\ax\rightarrow \bx$ is a factor map, then for any $f_1,...,f_n\in C(\bx)$
    \[\phi^{\ax}\bigl(f_1\circ F,...,f_n\circ F\bigr)\leq \phi^{\bx}(f_1,...,f_n). \]
\end{prop}

Why do we jump through all these hoops to ensure that our predicates are uniformly continuous? It's so that they commute with ultralimits. This is the content of \L o\'s's theorem for continuous logic. The proof is again a simple induction and can be found in \cite[Theorem 5.4]{Survey}. 

\begin{thm}[Continuous \L o\'s's theorem] \label{thm:los}
    For any $I$-indexed family of actions $\ip{\ax_i:i\in I}$, ultrafilter $\s U$ on $I$, elements of the ultraproduct $[f_1],...,[f_n]\in C:=\lim_{\s U} C(\ax)$, and $n$-place predicate $\phi$,
    \[\phi^{C}([f_1],...,[f_n])=\lim_{\s U} \phi^{\ax_i}(f_1,...,f_n).\]
\end{thm}

\L o\'s's theorem captures most of the arguments we have made about ultraproducts. But, we need to take some care to wrangle questions into the form of predicates. For instance, in the previous section we relied heavily on the observation that $f\geq g$ if and only if $\inf_w(|f-g-w^2|)=0$.

As a consequence, in an ultraproduct $[f]\geq [g]$ if and only if we can find $f'(i)\geq g'(i)$ so that $[f]=[f']$ and $[g]=[g']$. Continuous model theory highlights the importance of this kind of definable relation--relations which pass back and forth between the limit--and gives a straightforward test for definability.

\begin{dfn}\label{dfn:definable}
   Consider a family of sets $P^{\ax}$ where, for each $\Gamma$-flow $\ax$, $P^{\ax}\subseteq C(\ax)^n$. We say that $P^{\ax}$ is \textbf{definable} if there is a first-order formula $\phi$ and uniformly continuous $\alpha, \beta:\R\rightarrow\R$ so that
    \[P^\ax=\{x: \phi^\ax(x)=0\}\]
    \[(\forall\epsilon>0)(\exists \delta>0)(\forall x)\; \phi^\ax(x)<\delta\rightarrow d(x, P^\ax)\leq\epsilon.\]
    We say $P^\ax$ is $\Sigma_1$\textbf{-definable} if it is definable by some $\phi$ which is existential.

    We suppress $\ax$ in notation when the dependence on $\ax$ is clear.
\end{dfn}

For example, the set of non-negative functions
\[P=\{f\in C(\ax): (\forall x)\;f(x)\geq 0\}=\{f\in C(\ax): (\exists g)\; f=g^2\}\] is definable by $\phi(f)=\inf_{|g|\leq \max(|f|,1)} |f-g^2|.$ Clearly $P=\{f: \phi(f)=0\}$. And, if $\phi(f)<\epsilon$, then for some $g$, $|f-g^2|<\epsilon$ and $g\in P$.

The key point of definability is the following proposition. 

\begin{prop}
    A set $P\subseteq C(\ax)$ is definable iff, for any first-order predicate $\psi(x,y)$,
    \[\inf_{x\in P}\psi(x,y)\] is also first-order. In fact, if $P$ is $\Sigma_1$-definable and $\psi$ is existential, then $\inf_{x\in P}\psi(x,y)$ is also existential.
\end{prop}

So, in particular, if $P$ is definable, and $P^{\lim_{\s U} C(\ax_i)}$  is nonempty, then $P^{\ax_i}$ is nonempty for a $\s U$-large set of indices $i$. The proof can be found in \cite[Theorem 9.17]{Survey}. The basic idea is that for any $\psi$, we can reweight by $\phi$ so that $\inf_{x\in P} \psi(x)=\inf_{x} \psi(x)+\phi(x)$.

A sentence, or zero-place predicate, gives a numerical invariant attached to a flow. We essentially showed in the last section that these invariants are enough to recover the covering dimension of the underlying space and tell whether the action is free. The ensemble of all of these invariants is called the first-order theory of the action.

\begin{dfn}
    Let $\s L$ be the set of first-order sentences. For a $\Gamma$-flow $\ax$, its\textbf{ first-order theory} is 
    \[Th(\ax):=\ip{\phi^{\ax}: \phi\mbox{ a sentence}}\in \R^{\s L}.\]

    We say $\ax, \bx$ are \textbf{elementary equivalent }if $Th(\ax)=Th(\bx)$ and write $\ax\equiv\bx$

    Similarly, let $\Sigma_1$ be the set of existential sentences and define the \textbf{existential theory }of $\ax$ analogously:
    \[Th_{\Sigma_1}(\ax)=\ip{\phi^{\ax}: \phi\in \Sigma_1}\in \R^{\Sigma_1}.\] 
\end{dfn}

Since we allow any continuous connectives, $\s L$ is uncountable. We could restrict our language to use only countable dense set of continuous connectives, so the relations $(\equiv)$ and $(\equiv_{\Sigma_1})$ are in fact smooth Borel equivalence relations. See \cite[Section 3.11]{Survey} for more details.

We'll end this section with two general results from model theory of direct importance to this paper. 

First, the L\"owenheim--Skolem theorem says every structure has a substructure with the same first order theory \cite[Proposition 7.3]{Survey}. The sense of ``same first-properties" is a little stronger than elementary equivalence. In our context, we get the following:

\begin{thm}[Continuous L\"owenheim--Skolem theorem]
    For any flow $\ax$, there is a metrizable factor $\ax_0\leq \ax$ with factor map $q$ so that, for any $f_1,...,f_n\in C(\ax_0)$ and predicate $\phi$,
    \[\phi^{\ax_0}(f_1,...,f_n)=\phi^{\ax}(f_1\circ q,...,f_n\circ q).\]
\end{thm}

In practice, this theorem lets us reduce to studying flows on metrizable spaces, and also lets us construct interesting dynamics on metric spaces out of more exotic flows.

Finally, to connect all this theory directly to weak containment, we can invoke a general result on substructures of ultra-products. A general theorem says that $Th_{\Sigma_1}(\s M)$ dominates $Th_{\Sigma_1}(\s N)$ if and only if $\s M$ embeds into an ultrapower of $\s  N$. A continuous version of theorem doesn't seem to be readily accessible in the literature. For a proof in the setting of classical model theory, see \cite{Marker}. We'll only prove the special case below, and we delay the proof until Section \ref{sec: equivalences}.

\begin{thm}
    For any $\Gamma$-flows $\ax$ and $\bx$, $\ax$ weakly contains $\bx$ if and only if $Th_{\Sigma_1}(\bx)$ dominates $Th_{\Sigma_1}(\ax)$, that is
    \[\bx\preccurlyeq\ax\;\lra \;(\forall \phi\in \Sigma_1)\; \phi^{\ax}\leq \phi^{\ax}.\]
\end{thm}

So, if $\ax$ weakly contains $\bx$, then $\ax$ has more witnesses to existential facts, and $\phi^\ax$ is closer to true than $\phi^\bx$ whenever $\phi\in \Sigma_1$.

\subsection{The space of actions}

For a compact metric space $X$, we can endow the group of homeomorphisms with the topology of uniform convergence. This makes $\homeo(X)$ a Polish group.\footnote{A technical point: $\homeo(X)$ is $G_\delta$ in the space of continuous function on $X$. So, it is a Polish space, but we need to choose a slightly different metric from the usual one to witness this.} The space of actions naturally lives in $\homeo(X)^\Gamma$:

\begin{dfn}
    For a compact metric space $X$, $\homeo(X)$ is the Polish group of homeomorphisms of $X$ with the topology induced by the complete metric \[d(\rho,\tau)=\sup_{x\in X} d(\rho(x),\tau(x))+d(\rho\inv(x),\tau\inv(x)).\] 

    The\textbf{ space of actions }of a countable group $\Gamma$ on a compact metric space $X$ is the space
    \[\Axn(\Gamma, X)=\bigl\{\ax\in \homeo(X)^\Gamma: \ax(\id)=\id\;\&\;(\forall \gamma,\delta\in \Gamma) \;\ax(\gamma\delta)=\ax(\gamma)\ax(\delta)\bigr\}.\]

    The group $\homeo(X)$ acts on $\Axn(\Gamma, X)$ by \textbf{conjugacy}:
    \[(\rho\cdot \ax)(\gamma)=\rho\circ \ax(\gamma)\circ\rho\inv.\] We call the orbits of the conjugacy action \textbf{conjugacy classes}.
\end{dfn}

There are alternative parametrizations of actions: as subshifts of $X^\Gamma$, as actions by isometry on $C(X)$, as actions on Boolean algebras in the case where $X$ is a zero-dimensional space. For most purposes, these turn out to be equivalent. See, for instance,\cite{hochman}.

Several authors have studied the generic behavior of flows in $\Axn(\Gamma, \s C)$, especially in recent years. For instance, Doucha has characterized when there is a generic conjugacy class, and Iyer and Shinko have shown that a generic action is hyperfinite when $\Gamma$ has finite asymptotic dimension \cite{DouchaGeneric, IyerShinko}.

The space of actions $\Axn(\Gamma, X)$ is closed in $\homeo(X)^\Gamma$, and so it is a Polish space in itself. Equipped with the conjugacy action this becomes a (non-compact) Polish dynamical system. The topological Scott analysis gives a general method for extracting useful invariants in this situation \cite{GaoIDST}. The first step to computing these invariants (which is the only step we'll take in this article) is to understand the closure of orbits.   

\begin{dfn}
    Write $[\ax]$ for the conjugacy class of $\ax\in \Axn(\Gamma, X)$. Say that $\ax$ is \textbf{approximately conjugate} to $\bx$ if $\bx\in \overline{[\ax]}$.
\end{dfn}

Be careful! This is not a symmetric relation. It is possible that $\bx$ is approximately conjugate to $\ax$ but not vice-versa; $\ax$ could have a strictly larger ambit than $\bx$. As an example, every order preserving action in $\Axn(\Z, [0,1])$ is approximately conjugate to the trivial action, but the trivial action is only approximately conjugate to itself.

Unlike weak containment, approximate conjugacy only makes sense as a relation between flows on the same underlying space. And, even for simple spaces like $[0,1]$ approximate conjugacy is not equivalent to weak containment. Nonetheless the two notions are intimately related.

For Cantor actions, weak equivalence is the same as approximate conjugacy.

\begin{thm}
    For $\ax,\bx\in \Axn(\Gamma,\s C)$, $\ax$ is approximately conjugate to $\bx$ if and only if $\bx\preccurlyeq \ax$.
\end{thm}

We will delay the proof until the next section. For actions on general compact metric spaces, we only have one implication. Approximate conjugacy is a finer relation.

\begin{thm} \label{thm: one direction approximate conjugacy}
    For any compact metric space $X$ and $\ax,\bx\in \Axn(\Gamma, X)$, if $\ax$ is approximately conjugate to $\bx$, then $\bx\preccurlyeq\ax$.
\end{thm}
\begin{proof}

Let $d_H$ be the Hausdorff distance on closed subsets of $X$. For $W\Subset \Gamma, n\in \N$, and a collection of $(W,n)$-patterns $P\subseteq \s P([n]\times W)$, define
    \[U(P):=\{\ax\in \Axn(\Gamma, X): (\exists \vec F)\; p_W(\vec F,\ax)=P\}.\] It suffices to show that $U(P)$ is open for $P\subseteq \s P([n]\times W).$

    Suppose $\ax \in U(P)$ as witnessed by $\vec F=(F_1,...,F_n)$. Fix some separation $\delta>0$ so that, for any $\sigma\in \s P([n]\times W)$ with $\sigma\not\in P$ and for any $x\in X$ 
    \[\max_{(i,\gamma)\in \sigma} d(\gamma\cdot x, F_i) >10\delta.\] And, suppose $d(\bx(\gamma),\ax(\gamma))<\delta$ for all $\gamma\in W$. 
    
    Then, for any $\vec G$ close enough to $\vec F$, the there are no new local patterns in $p_W(\vec G, \bx)$ outside of $P$. More precisely, if $d_H(\vec G,\vec F)<\delta$, then for any $\sigma\subseteq \s P([n]\times W)$ with $\sigma\not\in P$
    \[\bigcap_{(i,\gamma)\in \sigma} \gamma\cdot_{\ax} F_i=\emptyset \;  \Rightarrow \;  \bigcap_{(i,\gamma)\in \sigma} \gamma\cdot_{\bx} G_i=\emptyset.\] 
    
   Now, we can find some $\vec G$ with $d_H(\vec G, \vec H)$ that witnesses all of the local patterns of $P$ in $\bx$ as follows. Define \[G_i:=\{x\in X: d(x, F_i)\leq\delta/2\}.\] Then, for any $\sigma\in P$ \[\emptyset\not=\bigcap_{(i,\gamma)\in \sigma} \gamma\cdot_{\ax} F_i\subseteq \bigcap_{(i,\gamma)\in \sigma} \gamma\cdot_{\bx} G_i.\] (Note that if $\vec F$ is a clopen cover, we don't need to expand $\vec F$ to $\vec G$ since the set of witnesses to $\bigcap_{(i,\gamma)\in \sigma} \gamma\cdot F_i\not=\emptyset$ is open.) So, $\sigma\in P$. Thus, $\vec G$ witnesses that $\bx\in U(P)$.
\end{proof}

However, the converse can fail. In fact, even if $\ax$ factors onto $\bx$, we cannot conclude that there is an approximate conjugacy:

\begin{prop} \label{prop:badfactor}
    Let $\sigma\in \homeo([0,1])$ be the involution $\sigma(x)=1-x$. Then, the $\Z$-flow $\ax(n)=\sigma^n$ factors onto the trivial $\Z$-flow $\bx(n)=\id$. But, $\bx\not\in\overline{[\ax]}$.    
\end{prop}
\begin{proof}
    The tent map \[f(x)=\left\{\begin{array}{cc} 2x & x\leq 1/2 \\
    1-2x & x\geq 1/2\end{array}\right.\] is a factor map. And, any conjugate of $\sigma$ satisfies $\rho\cdot\sigma(0)=1$, so $\id\not\in \overline{[\ax]}$.
\end{proof}

\subsection{Equivalences} \label{sec: equivalences}

Let us at last prove that all of these notions coincide.

\begin{thm}
    Fix flows $\ax:\Gamma\curvearrowright X$ and $\bx:\gamma\curvearrowright Y$. The following are all equivalent:
    \begin{enumerate}
        \item $\bx\preccurlyeq\ax$
        \item $(\forall \phi\in \Sigma_1)\; \phi^\bx\geq \phi^\ax$
        \item For some ultrafilter $\s U$, $\ax^\s U$ factors onto $\bx$.
    \end{enumerate}
\end{thm}
\begin{proof}
    $2\Rightarrow 1)\;$ We use level sets to extract closed covers from $\Sigma_1$ predicates. Fix $W\Subset \Gamma$ and say that a list of functions $\vec f=(f_1,...,f_n)$ represents a set of $W$-local patterns $P\subseteq \s P([n]\times W)$ if 
    \begin{itemize}
        \item $\bf 0\leq f_i\leq \bf 1$, and 
        \item For any $W\Subset \Gamma$ and $\sigma\in \s P([n]\times W),$ \[\bigl|\prod_{(i,\gamma)\in \sigma} \gamma\cdot f_i\bigr|=1\;\lra\;\bigl|\prod_{(i,\gamma)\in \sigma} \gamma\cdot f_i\bigr|\geq \epsilon\;\lra\; \sigma\in P.\]
    \end{itemize} A set of local patterns $P$ is represented by a list of functions in $C(\bx)$ if and only if $P\in P_{W,n}(\ax)$: if $P=p_W(\vec C, \ax)$, $f_i\inv[\{1\}]=C_i$, and $\im f_i\subseteq[0,1]$, then $\vec f=(f_1^N, ..., f_n^N)$ represents $\vec C$ for large enough $N$.

    The set of functions that represent $P$ is $\Sigma_1$-definable since pointwise inequalities and norm inequalities are all $\Sigma_1$-definable. So, if $\phi^\ax\leq \phi^\bx$ for all $\Sigma_1$ sentences $\phi$, then for any closed cover $\vec C$ of $Y$ there is a list of functions $f_1,...,f_n\in C(\ax)$ which represents $p_W(\vec C, \bx)$. 

    Define $B_i=f_i\inv[\{1\}].$ Then, for $\sigma\in \s P([n]\times W)$
    \begin{align*} 
        \bigcap_{(i,\gamma)\in \sigma} \gamma\cdot B_i\not=\emptyset &\lra \bigl|\prod_{(i,\gamma)\in \sigma} \gamma\cdot f_i\bigr|>\epsilon\\
        &\lra \bigcap_{(i,\gamma)\in \sigma} \gamma\cdot C_i\not=\emptyset.
    \end{align*} Thus $p_W(\vec B,\ax)=p_W(\vec C,\bx)$. So $\bx\preccurlyeq\ax$.

    $3\Rightarrow 2)\;$ This follows immediately from \L o\'s's theorem (Theorem \ref{thm:los}) and Proposition \ref{prop:embedding}.

    $1\Rightarrow 3)\;$ We use the layer-cake decomposition to extract functions from closed covers. First, note that, by the L\"owenheim--Skolem theorem, we can assume that $Y$ is separable.
    
    Now, say that $\phi: A\rightarrow C(\ax)$ is an $(W,\epsilon)$-isomorphism if, for any $f,g,h\in A$
    \[|\phi(f)-\phi(g)|\approx_\epsilon |f-g|,\quad \quad |\phi(f)+\phi(g)-\phi(h)|\approx_\epsilon |f+g-h|\]
    \[|\gamma\cdot \phi(f)-\phi(g)|\approx_\epsilon |\gamma\cdot f-g|,\quad |\phi(fg)-\phi(h)|\approx_\epsilon |fg-h|.\] (Here, $x\approx_\epsilon y$ means $|x-y|<\epsilon)$. A standard compactness argument gives the following:

    I claim that if, for any $W\Subset\Gamma$, $\epsilon>0$, and  $D\Subset C(\bx)$, there is a $(W,\epsilon)$-isomorphism from $D$ to $C(\ax)$, then there is an embedding of $C(\bx)$ into $C(\ax^{\s U})$.

       Fix $D_n\Subset C(\bx)$ so that $D_n\subseteq D_{n+1}$ and $\bigcup_n D_n$ is dense in $C(\bx)$ (using the separability of $C(\bx)$), and fix $W_n\Subset \Gamma$ with $\bigcup_n W_n=\Gamma$. For each $n$, fix a $(W_n,2^{-n})$-isomorphism $\phi_i: D_i\rightarrow C(\ax)$. Define $\Phi: C(\bx)\rightarrow C(\ax^{\s U})$ as follows: given $f\in C(\bx)$, fix $f(i)\in D_i$ and let 
    \[\Phi(f)=\lim_{\s U} \phi_i(f(i)).\] Then, for any $f,g$,
    \[|\Phi(f)-\Phi(g)|=\lim_{\s U} |\phi_i(f(i))-\phi_i(g(i))|=\lim_{\s U} |f(i)-g(i)|=|f-g|.\] Similarly, $|\gamma\cdot \Phi(f)-\Phi(g)|=|\gamma\cdot f-g|.$ In particular,
    \[|\gamma\cdot \Phi(f)-\Phi(\gamma\cdot f)|=|\gamma\cdot f-\gamma\cdot f|=0.\] So, $\gamma\cdot \Phi(f)=\Phi(\gamma\cdot f)$. By the same argument $\Phi(fg)=\Phi(f)\Phi(g)$ and $\Phi(f+g)=\Phi(f)+\Phi(g)$. So, $\Phi$ is an equivariant isometric embedding as desired.

    So, suppose $\bx\preccurlyeq\ax$, and fix $D\Subset C(\bx)$, $W\Subset \Gamma$ and $\epsilon>0$; we want a $(W,\epsilon)$-isomorphism from $D$ to $C(\ax)$. Define $\phi:D\rightarrow C(\ax)$ as follows: Pick $N\in \N$ with $N>\max_{f\in D} |f|/\epsilon$. For $f\in D$ and $j\in \{-N,...,N\}$, let $C_{f,j}=f\inv[(j-1)/N,j/N]$, and pick $\vec D$ a closed cover of $X$ with $p_W(\vec D,\ax)=p_W(\vec C,\bx).$ Pick $[0,1]$-valued functions $\ell_{f,j}$ with $\ell_{f,j}\equiv 1$ on $\bigcup_{i<j} D_{f,i}$ and $\ell_{f,j}\equiv 0$ on $\bigcup_{i>j} D_{f,i}$. Finally, set \[\phi(f)=-N\bf 1+\sum_{j=-N}^{N} \ell_{f,j}.\]
    Now, we have to compute some norms. Fix $f,g\in D$. We have
    \[|\phi(f)-\phi(g)|\approx_{\epsilon} \max\left\{\frac{|i-j|}N: C_{f,i}\cap C_{g,j}\not=\emptyset\right\} \approx_{\epsilon} |f-g| \] and \[|\gamma\cdot \phi(f)-\phi(g)|\approx_{\epsilon} \max\left\{\frac{|i-j|}N: \gamma\cdot_{\bx} C_{f,i}\cap C_{g,i}\not=\emptyset\right\}\approx_{\epsilon} |\gamma\cdot f-g|.\] The other approximations follow similarly.
\end{proof}

And for approximate conjugacy in the setting of Cantor flows, we have:

\begin{dfn}
    For any continuous labelling $f:\s C\rightarrow [k]$ and $W\subseteq \Gamma$ finite, define a set
    \[U(f, W, \bx):=\bigl\{\ax: \s P(f, W, \ax)=\s P(f, W, \bx)\bigr\}\subseteq \Axn(\Gamma, \s C).\]
\end{dfn}

\begin{prop}
    The sets $U(f, W, \bx)$, as $f$ ranges over continuous labellings of $\s C$ and $W\Subset\Gamma$ over finite windows, form a neighborhood basis of $\bx\in \Axn(\Gamma, \s C)$.
\end{prop}
\begin{proof}
    That $U(f, W, \bx)$ is open follows from the proof of Theorem \ref{thm: one direction approximate conjugacy} about approximate conjugacy in general. So, we just have show that each neighborhood of $\bx$ contains some $U(f, W,\bx)$.

    A basic open set in $\Axn(\Gamma, \s C)$ determines, up to uniform tolerance $\epsilon$, the functions $\bx(\gamma)$ for $\gamma$ in some finite window $W\Subset\Gamma$. So, fix a window $W$ and a tolerance $\epsilon>0$. Let $f$ be a labelling of $\s C$ so that all the sets of the form $f\inv[i]$ or $\gamma\cdot_{\bx} f\inv[i]$ for $\gamma\in W$ have diameter at most $\epsilon/10$.  
    
    Suppose $\ax\in U(f, W, \bx)$, and fix $x\in \s C$ and $\gamma\in W$. We want to show that $d(\gamma\cdot_{\ax} x,\gamma\cdot_{\bx} x)<\epsilon.$ There is some $y\in \s C$ so that $f(y)=f(x)$ and $f(\gamma\cdot_{\ax} x)=f(\gamma\cdot_{\bx} y)$ since these actions have the same local patterns in $f$. Since $f(\gamma\cdot_{\ax} x)=f(\gamma\cdot_{\bx} y)$, $d(\gamma\cdot_{\ax} x, \gamma\cdot_{\bx} y)\leq \epsilon/10. $  And, since $f(x)=f(y),$ $d(\gamma\cdot_{\bx} x,\gamma\cdot_\bx y)\leq \epsilon/10$. So, 
    \[d(\gamma\cdot_{\ax} x, \gamma\cdot_{\bx} x)\leq d(\gamma\cdot_{\ax} x, \gamma\cdot_{\bx} y)+d(\gamma\cdot_{\bx} y,\gamma\cdot_{\bx} x)<\epsilon\]
\end{proof}

Our last characterization of weak containment follows easily.

\begin{thm}\label{thm: approximate conjugacy}
    For any Cantor flows $\ax,\bx\in \Axn(\Gamma, \s C)$, 
    \[\bx\in \overline{[\ax]}\;\lra \; \bx\preccurlyeq\ax.\] (In words: $\ax$ is approximately conjugate to $\bx$ iff $\bx$ is weakly contained in $\ax$.)
\end{thm}
\begin{proof}
    The left to right direction is a special case of Theorem \ref{thm: one direction approximate conjugacy}. So, suppose that $\bx\preccurlyeq\ax$.

    Consider any open neighborhood $U=U(f, W, \bx)$ of $\bx$; we want to find some conjugate of $\ax$ in $U$. By definition of weak containment, there is some continuous labelling $g$ so that $\s P(g, W, \ax)=\s P(f, W, \bx).$ And, since $\s C$ is ultrahomogenous, there is some $\rho\in \homeo(\s C)$ so that $g=f\circ\rho\inv.$ Thus, $\s P(f, W,\bx)=\s P(f, W,\rho\cdot\ax).$ So, $\rho\cdot \ax\in U$ as desired.
\end{proof}

\section{Tops, bottoms, and anti-chains}

So, what can topological weak containment tell us about groups and their actions? We'll give some suggestive partial answers in the section. In particular we'll investigate structural features of the weak containment order and see that they correspond to a diverse array of dynamical, combinatorial, and group-theoretic phenomena.

We will focus our attention on zero-dimensional flows, and especially on $4$ subclasses:
\begin{dfn}
    For a fixed group $\Gamma$,
    \[\kfree(\Gamma) :=\{\ax\in \Axn(\Gamma,\s C): \ax\mbox{ is free}\}\]
    \[\kmix(\Gamma): = \{\ax\in \Axn(\Gamma, \s C): \ax\mbox{ is mixing}\}\]
    \[\kmin(\Gamma):= \{\ax\in \Axn(\Gamma,\s C): \ax\mbox{ is minimal}\}\]
    \[\kfin(\Gamma):=\{\ax: \ax\mbox{ is free, and }(\exists \ax_i\mbox{ actions on finite set})\;\ax=\colim_{\s U}\ax_i\}.\] We'll supress $\Gamma$ in notation when it is clear from context.
\end{dfn}
Note that $\kmix$, $\kmin$, and $\kfin$ are not invariant under weak equivalence. Also, by Proposition \ref{prop: covering dimension}, $\kfin$ contains only actions on zero-dimensional spaces. 

 Let us point out that the analogous class to $\kfin$ in the context of ergodic theory is well studied. It corresponds to so-called local-global limits of finite graphs. Many deep questions remain open even for $\Gamma=F_2$: Is every free pmp action a measure theoretic local-global limit of finite actions? is the Bernoulli shift? Positive or negative answers to these questions would have deep ramifications for combinatorics. For instance, if the Bernoulli shift is a local-global limit, then we would have a powerful machine for producing counterexamples in extremal combinatorics. And if not, then there is some labelling that can be simulated on almost every finite $4$-regular finite graph of large enough girth but that cannot be approximated with a local algorithm. We will return to these questions in the topological setting later.

We consider three of the most basic order-theoretic questions: Is there a top? Is there a bottom? How big can antichains be?

\subsection{Top actions}

Which of these classes have top elements? And, what are these top elements like-- what dynamical properties do they have, and what local patterns do they admit? First, let's recall the definition of top:

\begin{dfn}
    For a class $\s K$ of actions, say that $\ax$ is a \textbf{top} element of $\s K$ if $\ax\in \s K$ and, for all $\bx\in \s K$, $\bx\preccurlyeq\ax$. If $\s K$ has a top element, we write $\ax_{\top}^\s K$ for a top element (which will only be unique up to weak equivalence). We abbreviate $\ax_\top=\ax^{\Axn(\Gamma, \s C)}_\top$.
\end{dfn}

These might also be called maximum elements, but then we would run into ambiguity about the word ``minimal" when we discuss bottom actions. All of the classes we have in our sights admit top elements:

\begin{thm}
    The classes $\Axn(\Gamma, \s C)$, $\kfree$, $\kmix$, $\kfin$, and $\kmin$ all have top elements.
\end{thm}
\begin{proof}
     The classes $\Axn(\Gamma, \s C)$, $\kfree$, $\kmix$, and $\kfin$ are all closed under countable products (up to isomorphism). So, if $\s K$ is any of these, we can let $\{\ax_i: i\in \omega\}$ be a countable dense set of actions in $\s K$. Then, $\prod_i\ax_i$ is (up to isomorphism) a top element of $\s K$.

     For $\kmin$, we can take $\ax_{\top}^{\kmin}$ to be a separable elementary equivalent factor of the universal minimal $\Gamma$-flow.
\end{proof}

Now, what are these top elements like? One answer is that they are generic:

\begin{prop}
    If $\s K\subseteq \Axn(\Gamma, \s C)$ is conjugacy-invariant and $G_\delta$, then $\s K$ has a top element if and only if the conjugacy action on $\s K$ is generically ergodic. 

    Further, $\ax\approx \ax_{\top}^\s K$ if and only if $\ax$ has a dense orbit.
\end{prop}
\begin{proof}
    This follows directly from Theorem \ref{thm: approximate conjugacy}
\end{proof}
\begin{cor}
    The conjugacy of $\homeo(\s C)$ is generically ergodic on 
    \begin{enumerate}
        \item $\Axn(\Gamma, \s C)$
        \item $\kfree(\Gamma)$, and
        \item $\kmin(\Gamma)$.
    \end{enumerate}
\end{cor} It turns out point $(2)$ is a little redundant since generic actions are free. The previous proposition is fairly trivial. Nonetheless, it suggests a useful approach to studying genericity in the space of actions. For many groups, there is not a comeager conjugacy class in $\Axn(\Gamma, \s C)$. Even when there is a comeager conjugacy class, it can be difficult to describe. But, weak and elementary equivalence classes are tractable approximations to conjugacy classes. The above proposition tells us there is always a comeager weak equivalence class and implies there is always a comeager elementary equivalence class (since elementary equivalence is smooth). Existential and first order theories are relatively easy to understand, and can give us a handle on the behavior of generic actions.

Another answer answer to the question ``What is $\ax_\top^\s K$ like?" is that top elements are easy to decorate. For instance, we showed earlier that being free is equivalent to admitting $3$-colorings of a certain kind. So, we have:

\begin{prop}
    If $\s K$ has a free element, then $\ax_{\top}^{\s K}$ is free.
\end{prop}
\begin{proof}
    Follows from Theorem \ref{thm: free closed up}.
\end{proof}
\begin{cor}
 A generic action, or generic minimal action, in $\Axn(\Gamma, \s C)$ is free.
\end{cor}

It follows that $\ax_\top\approx\ax_{\top}^\kfree$. We can also rule out topological transitivity easily:

\begin{prop} \label{prop: generic not top transitive}
   If $\Gamma$ is finitely generated, the actions $\ax_{\top}$ and $\ax_{\top}^{\kfin}$ cannot be topologically transitive.
\end{prop}
\begin{proof}
    Let $\s K=\Axn(\Gamma, \s C)$ or $\kfin$, and suppose $\Gamma=\ip{E}$ with $E$ finite. There is some action $\ax\in \s K$ with a labelling $f$ so that $f(\gamma\cdot_{\ax} x)=f(x)$ for $\gamma\in E$ and $0,1\in \im(f)$. So, there is such a labelling for $\ax_{\top}^\s K$, call it $g$. But then, $g\inv[0]$ and $g\inv[1]$ witness that $\ax_\top^\s K $ is not topologically transitive.
\end{proof}
\begin{cor}
    A generic action in $\Axn(\Gamma, \s C)$ is not topologically transitive when $\Gamma$ is finitely generated.
\end{cor}

Iyer and Shinko's proof of hyperfiniteness for generic actions of finite asymptotic dimension groups follows this pattern as well \cite{IyerShinko}.

\begin{prop}[Ess. Iyer--Shinko]
    If $\Gamma$ has finite asymptotic dimension, then $\ax_\top^{\kmin}$, $\ax_{\top}^{\kmix}$, and $\ax_{\top}$ are all hyperfinite.
\end{prop}
\begin{proof}[Proof sketch]
    Say $\operatorname{asdim}(\Gamma)=n$, and let $\bx_{r,k}$ be the space of $k$-bounded equivalence relations on $\Gamma$ so that each $r$-ball meets at most $(n+1)$-equivalence classes. Then, for all $r$ and all large enough $k$, $\bx_{r,k}$ is a shift of finite type, is mixing, and has some minimal subflow. So, for $\s K$ any class from the theorem statement, $\ax_{\top}^{\s K}$ has a map into $\bx_{r,n}$. Thus $\ax_{\top}^{\s K}$ has Borel asymptotic dimension at most $n$, so is hyperfinite by Conley--Jackson--Marks--Seward--Tucker-Drob \cite{BAD}.
\end{proof}
\begin{cor}
    For $\Gamma$ with finite asymptotic dimension, a generic action on $\s C$ and generic minimal action on $\s C$ are hyperfinite.
\end{cor}

For $\Gamma=\Z^2$ or $F_n$, the same argument gives that $\ax_{\top}^{\kfin}$ is hyperfinite.


As easy as these actions are to decorate, the general problem of determining what patterns they admit can be intractable. The following observation is essentially true by definition:

\begin{prop}
    For any $W\Subset\Gamma$ and $P\subseteq [k]^W$, $P=p_W(f, \ax_\top)$ if and only if $\Gamma$ admits labelling $f$ with $p_W(f, \rho)=P$ (here, $\rho:\Gamma\curvearrowright\Gamma$ is right translation).
\end{prop}
\begin{cor}
    For $\Gamma=\Z^2$, it is not computable to check if $P=p_W(f, \ax_\top)$.
\end{cor}
\begin{proof}
    One can easily reduce the Wang tiling problem to set of local patterns of $\Z^2$ labellings.
\end{proof}
\subsection{Bottom classes}

What about bottom elements? The definition of bottom is, of course, dual to the definition of top:

\begin{dfn}
    For a class of actions $\s K$, a \textbf{bottom} is an action $\bx_{\bot}^{\s K}\in \s K$ so that, for all $\ax\in \s K$, $\bx_{\bot}^{\s K}\preccurlyeq \ax$. We abbreviate $\bx_\bot^{\kfree}=\bx_\bot$
\end{dfn}

Constructing bottom elements is not as easy as constructing top elements. The product of two flows is an easy upper bound, but there is no easy lower bound. In terms of the dynamics of conjugacy, bottom elements correspond to minimal closed subflows. But, since $\Axn(\Gamma, \s C)$ isn't compact, we aren't guaranteed to have such subflows. 

Nonetheless, Elek showed that there is a bottom element among the free flows for any finitely generated group \cite[Theorem 2]{ElekQW}. We give a slight improvement of Elek's theorem below.

\begin{dfn}[Elek \cite{ElekQW}]
    For a group $\Gamma$, and $W\Subset \Gamma$ symmetric, let $\bx_W$ be the shift action on proper $|W|+1$-colorings of $\cay(\Gamma, W)$. That is,
    \[\bx^\Gamma_n:\Gamma\curvearrowright C_W\subseteq (|W|+1)^\Gamma\]
    \[C_W:=\bigl\{x\in(|W|+1)^\Gamma: (\forall \gamma\in \Gamma,\delta\in W\setminus\{\id\})\; x(\gamma)\not=x(\gamma\delta)\bigr\}.\] And, set $\bx_\bot$ to be any action on $\s C$ which is weakly equivalent to $\prod_{W\Subset\Gamma} \bx_W$.
\end{dfn}
Technically, we have overloaded the symbol $\bx_\bot$, but this action will turn out to be a bottom for $\kfree(\Gamma)$.

For any $W$ and any partial $(|W|+1)$-coloring $p$ of $\cay(\Gamma, W)$, a simple greedy algorithm can extend $p$ to a total coloring. It follows that $\bx_W$ has no isolated points, and we can identify $\bx_W$ with an element of $\Axn(\Gamma,\s C)$.

A standard continuous coloring argument says that any free flow has a map into $\bx_W$. Elek showed that, further, we make this map hit any desired clopen set, so $\bx_\bot$ is indeed a bottom element for $\kfree(\Gamma)$. We will argue further still that we can make $f[C]$ meet $D$ for any clopen $C,D$.

\begin{thm}
    For any finitely generated group $\Gamma$ and any free flow $\ax\in \Axn(\Gamma, \s C)$, $\ax\times\bx_\bot\approx \ax$. (So, in particular, $\bx_\bot\preccurlyeq \ax$.)
\end{thm}

\begin{proof}
    For the sake of readability, we'll show that $\ax\times \bx_W\preccurlyeq \ax$ for all $W\Subset\Gamma$ symmetric. The theorem follows by repeating this argument simultaneously for any finite number of $W$'s. 
    
    We'll use Proposition \ref{prop:SFT containment}. Fix a free Cantor flow $\ax$ and consider nonempty open rectangles $A_1\times B_1,..., A_m\times B_m\subseteq \ax\times \bx_W$. First, we'll build a map  $F\in \hom(\ax, \bx_W)$ so that $F[A_i]\cap B_i\not=\emptyset$. 
    
    We can find some finite window $W\Subset\Gamma$ and patterns $\sigma_1,...,\sigma_k\in (k_n)^{W}$ defining neighborhoods in the $B_i$. That is, if $x\res W=\sigma_i$ then $x\in B_i$. We'll define a local rule $f$ for the map $F$, meaning we will define $f$ and set $F(x)(\gamma)=f(\gamma\inv x).$ We want that, for all $x\in \s C$, $\gamma\in \Gamma$, and $\delta\in \Delta$, $F(x)(\gamma\delta)\not=F(x)(\gamma)$, and we want that, for each $i$, there is some $x\in A_i$ so that $F(x)(\gamma)=\sigma_i(\gamma)$ for each $\gamma\in W$. So, our requirements on $f$ are:
    \[(\forall \gamma\in W, x\in X)\; f(x)\not=f(\gamma\inv x)\]
    \[(\exists x\in A_i)(\forall \gamma\in W)\; f(\gamma\inv\cdot x)=\sigma_i(x).\]

     Since $\s C$ has no isolated points, we can shrink that $A_i$'s so that $\bigl(W\inv\cdot A_i\bigr)\cap A_j=\emptyset$ for $i\not=j.$ Define a partial coloring $f_0:\s C\rightarrow(|W|+1)$ by setting $f_0(x)=\sigma_i(\gamma)$ if $\gamma\cdot x\in A_i$ (i.e.~if $x=\gamma\inv\cdot y$ for some $y\in A_i$). Then $f_0$ is well-defined since we shrank the $A_i$'s, $f_0$ is clearly continuous on its domain, and $\dom(f_0)=\bigcup_{\gamma\in W} \gamma\cdot A_i$ is clopen.
    
     Now, we'll extend $f_0$ to a continuous map $f: \s C\rightarrow \bigl(|W|+1\bigr)$. Let $U_0,..., U_n\subseteq\s C\setminus \dom(f_0)$ be disjoint clopen sets so that $\bigl(\gamma\cdot_{\ax} U_i\bigr)\cap U_i=\emptyset$ for each $i\in[n]$ and $\gamma\in W\inv$ (such a family of clopen sets exists by free-ness of $\ax$ and by compactness). Then, define $f\supseteq f_0$ on $U_i$ by induction as follows: Given $f\res \dom(f_0)\cup U_0\cup...\cup U_{k-1}$ and $x\in U_k$, let \[f(x)=\min\{c\in |W|+1: (\forall \gamma\in W)\;f(\gamma\inv\cdot x)\not=c\}.\] Then $f$ is well-defined since there are $|W|+1$ options for $c$; and $f$ is continuous since $f\inv(c)$ is a Boolean combination of clopen sets.

     By construction, $f(x)\not= f(\gamma\inv\cdot x)$ for all $x\in \s C$ and $\gamma\in W$. And, if $x\in A_i$, then $f(\gamma\inv \cdot x)=f_0(\gamma\inv \cdot x)=\sigma_i(\gamma)$ as desired. So, $F(x)(\gamma)=f(\gamma\inv\cdot x)$ defines a map in $\hom(\ax, \bx_W)$ with $F[A_i]\cap B_i\not=\emptyset$.
     
    Finally, $x\mapsto (x,F(x))$ is an equivariant map from $\ax$ to $\ax\times\bx$ so that $f[\s C]$ meets each $A_i\times B_i$. Thus $\ax\times \bx_n\preccurlyeq \ax$. The other weak containment is clear.
\end{proof}

One of Elek's motivations for studying this actions was to find a so-called ``qualitative" analogue of a statistical theory of graph limits. Here is one simple example of this kind of phenomenon. 

\begin{dfn}
    For $m,n\in \Z$, let $\mathbf{c}_{m,n}$ be the translation action on the finite $m\times n$ torus:
    \[\mathbf{c}_{m,n}:\Z^2\curvearrowright \Z^2/(m\Z\times n\Z)\]
    \[(a,b)\cdot (x,y)=(x+a \;\mathrm{ mod }\;m,y+b\;\mathrm{mod}\;n).\]
\end{dfn}
\begin{prop}\label{prop: graph limit}
    Suppose $\bx_\bot$ maps into a shift of finite type $X$. Then, there is some $M,N\in \N$ so that, if $m>M$ and $n>N$, then $\mathbf{c}_{m,n}$ maps into $X$.
\end{prop}
\begin{proof}
    Suppose otherwise. Then, there is a pair of sequences $m(i),n(i)$ both tending to infinity so that $\mathbf{c}_{m(i),n(i)}$. Fix a nonprincipal ultrafilter $\s U$,  and consider $\ax=\colim_{\s U}\mathbf c_{m(i), n(i)}$. By \L o\'s's theorem, Theorem \ref{thm: approximate conjugacy}, $\ax$ does not map into $X$ does not map into $X$. This means $\ax\times \bx$ does not map into $X$ where $\bx$ is the trivial action on $\s C$. But, $\ax^\N$ is a free action on a perfect zero-dimensional space, so $\bx_\bot\preccurlyeq \ax^\N$. But this contradicts the fact that $\bx_\bot$ maps into $X$.
\end{proof}

Seward and Bernshteyn (independently) gave a somewhat more natural presentation of this bottom element as a non-compact action. They prove the following:

\begin{thm}[{\cite[Theorem 1.11]{ContinuousCombo}}]
    For any countable $\Gamma$, every free flow weakly contains the free part of the Bernoulli shift $F(2^{\Gamma})$.
\end{thm}
\begin{cor}
    Any large enough metrizable factor of the Cech--Stone compactification, $\hat F(2^\Gamma)$, is a bottom for $\kfree$.
\end{cor}

Combining this with a connection between descriptive combinatorics and local algorithms, they exactly characterized the local patterns of $\bx_\bot$. 

\begin{cor}[{\cite[Theorem 1.15]{ContinuousCombo}}]
    The following are equivalent for a shift of finite type $S$:
    \begin{enumerate}
        \item There is a some $f\in \hom(F(2^\Gamma), S)$
        \item For any free $\ax\in \Axn(\Gamma, \s C)$, there is a homomorphism $f\in \hom(\ax, S)$
        \item (if $\Gamma$ is locally finite) there is an $O(\log^*(n))$-round local algorithm for $S$-labelling large enough graphs which locally resemble $\Gamma$.
    \end{enumerate}
\end{cor}

For example, this means that labellings of finite tori given by  Proposition \ref{prop: graph limit} can all be constructed by running an efficient local algorithm on finite tori. There is a similar correspondence between large girth $2n$-regular graphs and $F_{2n}$. The notion of local resemblence is spelled out in detail in Bernshteyn's paper.



What about the other classes mentioned in the introduction to this section? Elek showed that his bottom element can be chosen to be minimal and mixing, so \[\bx_{\bot}^{\kfree}\approx \bx_{\bot}^{\kmin\cap\kfree}\approx \bx_{\bot}^{\kmix\cap\kfree}.\] The story with $\kfin$ is more complicated.

If $\bx_{\bot}^{\kfree(\Gamma)}\approx \bx_{\bot}^{\kfin(\Gamma)}$, then, for instance, every set of local patterns than can be realized on all large enough finite tori can be constructed by an efficient local algorithm. For $\Z^2$, this isn't the case: every sequence of finite actions of $\Z^2$ admits some local pattern which cannot be produced locally.

\begin{thm}
    There is no bottom element for $\kfin(\Z^2)$.
\end{thm}
\begin{proof}
    For a prime $p$, let $\ax_p=\colim_{U} \mathbf{c}_{p^n,p^n}.$

    Suppose towards contradiction that $\bx=\lim_{\s U} \bx_i$ is a bottom element for $\kfin$. By considering local patterns like those in Proposition \ref{prop: generic not top transitive}, we can see that $\s U$-almost every $\bx_i$ is transitive. So, without loss of generality, $\bx_i$ is a finite torus $\mathbf{c}_{n_i,m_i}$ for some sequence of integers $n_i,m_i\rightarrow\infty$. 
    
    Passing to a $\s U$-large subset, we may assume that there is some $0\leq j<3$ so that, for all $i$, $n_i\equiv j\mod 3$ and $m_i>3$. Further, we may assume $n_i<m_i$ for all $i$ and $\max(m_i-n_i, 2)$ is constant. Let $W=\{\id,\vec e_1, \vec e_2,\vec e_1+\vec e_2, -\vec e_1, -\vec e_2,-\vec e_1-\vec e_2\}$. Consider $f$ as pictured in Figure \ref{pic: tiling}. For any $i$, $p_{W}(f,\mathbf{c}_{n_i,m_i})=P$ is the same. And, any periodic labelling with $W$-pattern $P$ has horizontal period divisible by $3k+j$ for some $k$.
    
    If $j\not=3$ $P_{A, W}(f, \bx)$ cannot be replicated by a labelling of $\ax_3$ since the horizontal period of any labelling of $\ax_3$ must be a power of $3$. So $\ax_3\not \preccurlyeq\bx$. And, if $j=3$, then $f$ can't be replicated by a labelling of $\ax_7$. So, in either case $\bx$ is not a bottom.

    \begin{figure}
        \centering
        \includegraphics[width=0.75\linewidth]{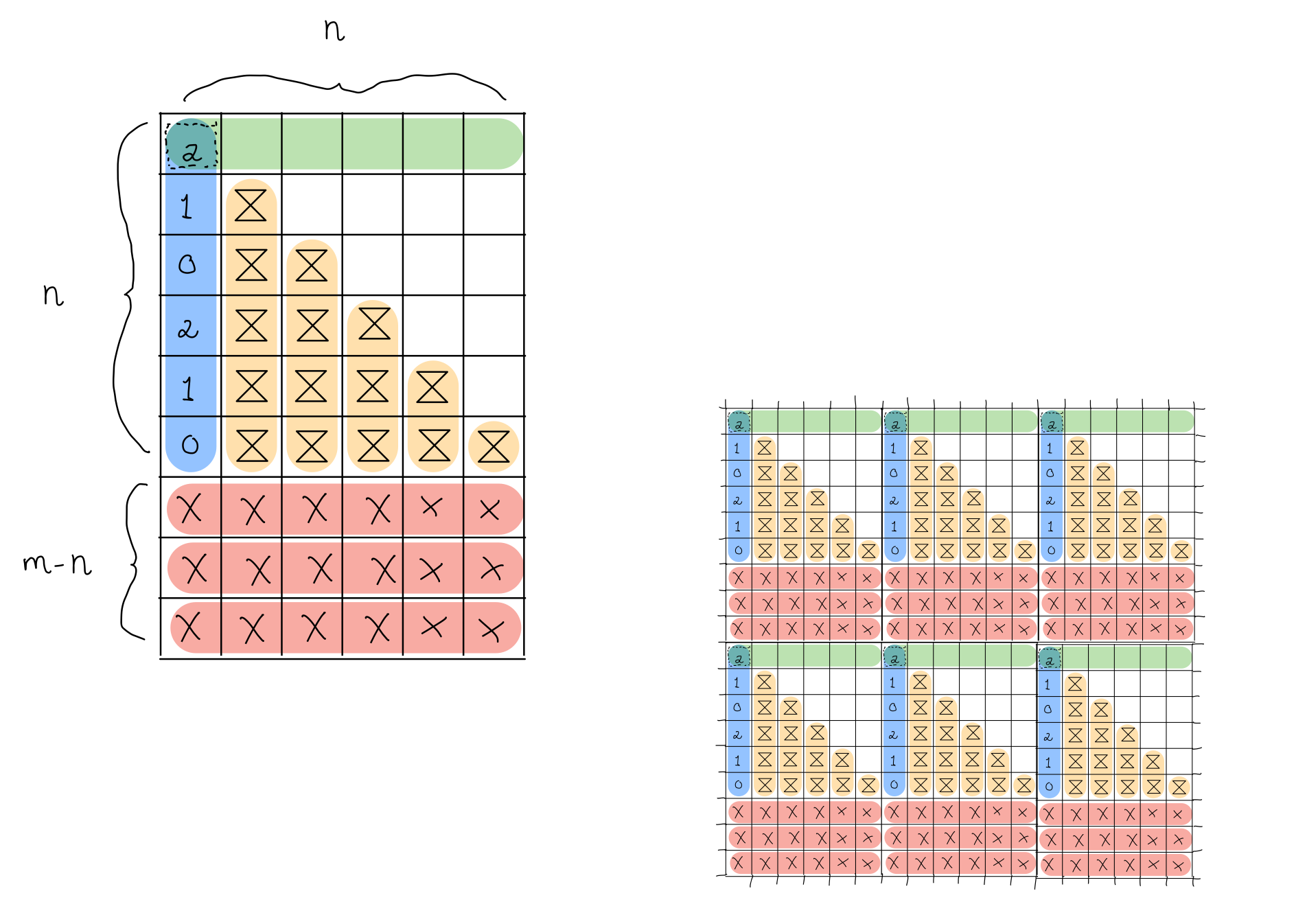}
        \caption{Any periodic tiling with the same local pattern as the above must have period $2$ mod $3$ in the $x$-direction.}
        \label{pic: tiling}
    \end{figure}
    
    \end{proof}

\subsection{Antichains}

It's natural to ask how large we can make antichains in the weak containment order for a given group. The answer seems to reflect some large-scale geometry of the group.

\begin{dfn}
    Say that a set of actions, $A$, is an antichain if 
    \[(\forall \ax,\bx\in A) \;\ax\not\preccurlyeq\bx\mbox{ and }\bx\not\preccurlyeq\ax.\]

    For a class of actions $\s K$, the \textbf{width} of $\s K$ is
    \[\wdth(\s K):=\sup\bigl\{|A|:A\subseteq \s K, \;A \mbox{ an antichain}\}.\]
\end{dfn}

Since the relation $\preccurlyeq$ is Borel, a theorem of Harrington, Marker, and Shelah implies that when $\s K$ is Borel $\wdth(\s K)$ is either countable or $|\R|$ \cite{HarringtonMarkerShelah}.

\begin{prop} \label{prop: z antichain}
    $\wdth\bigl(\kfree(\Z)\bigr)=|\R|$.
\end{prop}
\begin{proof}
For $n\in \N$, let $\mathbf c_n$ be the translation action of $\Z$ on $\Z/n\Z$. For $S$ an infinite set of primes, define \[\ax_S:= \prod_{p\in S} \mathbf c_p\quad \mbox{ and }\quad A:=\{\ax_S: S\mbox{ an infinite set of primes}\} .\] Clearly each $\ax_S$ is free and conjugate to an action on $\s C$. It remains to check that $A$ is an antichain.

If $p\in S$, then $\mathbf c_p\preccurlyeq\ax_S$. Suppose $p\not\in S$. I claim that $\hom(\ax_S,\mathbf c_p)=\emptyset.$ Since $\mathbf c_p$ is a shift of finite type, this implies $\mathbf c_p\not\preccurlyeq \ax_S$, and the main result follows. Suppose toward contradiction that $f\in \hom(\ax_S, \mathbf c_p)$. By continuity, $f$ only depends on finitely many coordinates, so induces a map from $\prod_{q\in F} \mathbf c_q$ to $\mathbf c_p$ for some $F\Subset S$. But this is impossible as $p$ does not divide $\prod F$.
\end{proof}

Note that what we produce is stronger than an antichain: for each $\ax,\bx\in A$, there is a shift of finite type $X$ so that $X\preccurlyeq \ax$ but $\bx$ has no maps whatsoever into $X$.

What about more general groups? If we have an action of a subgroup $\Delta\subseteq \Gamma$, we can co-induce an action of $\Gamma$.

\begin{dfn}
    For $\Delta\subseteq\Gamma$ a subgroup and $\ax\in \Axn(\Delta, X)$, the \textbf{co-induced action} is the subshift 
     \[\coind_\Delta^\Gamma(\ax): \Gamma\curvearrowright X^\Gamma_\ax\subseteq X^\Gamma\] where
    \[X^\Gamma_\ax=\biggl\{x\in X^\Gamma: (\forall\delta\in \Delta,\gamma\in \Gamma) \;x(\gamma\delta)=\delta\inv\cdot_{\ax}\bigl (x(\gamma)\bigr)\biggr\}.\]
\end{dfn}

So, if $\ax:\Gamma\curvearrowright X$ is a subshift of $A^\Delta$, then $\coind_\Delta^\Gamma(\ax)$ is conjugate the subshift of $A^\Gamma$ where each $\Delta$ coset is labelled with an element of $X$. That is,
\begin{align*}\coind_\Delta^\Gamma(X) &= \{x\in X^\Gamma: (\forall \gamma\in \Gamma, \delta\in \Delta)\; x(\gamma\delta)=\delta\inv\cdot (x(\gamma))\} \\ & \cong \{ x\in A^{\Gamma\times \Delta}: (\forall \gamma\in \Gamma, \alpha,\beta\in \Delta)\; x(\gamma\beta, \alpha)=x(\gamma,\beta\alpha)\mbox{ and } x(\gamma,\cdot)\in X\} \\ & \cong \{ x\in A^\Gamma: (\forall \gamma\in \Gamma) \; (\delta\mapsto x(\gamma\delta))\in X\}.\end{align*}Co-induction is has the universal property of being adjoint to restricting an action (i.e.~inducing an action on a subgroup):
\[\hom(\ax\res \Delta, \bx)\cong\hom(\ax,\coind_{\Delta}^\Gamma(\bx)).\] This gives a recipe for mapping out of a coinduced action. We'll also need a lemma for mapping into one.

\begin{lem} \label{lemma:coinduction map}
    Fix $\Delta\subseteq\Gamma$ a subgroup. Then, any $\Delta$-action $\ax\in \Axn(\Gamma, \s C)$ maps into the restriction of its coinduction to $\Gamma$:
    \[\hom(\ax,\coind_\Delta^\Gamma(\ax)\res \Delta)\not=\emptyset.\] Moreover, if we want to map into the restriction of the free part, we have
    \[\hom(\ax\times \bx_\bot^2, F(\coind(\ax\times \bx_\bot^2))\res \Delta)\not=\emptyset.\]
\end{lem}
So, for example, if $\Delta$ is the $x$-axis in $\Z^2$, this lemma implies that we can take an element of $\bx_\bot^2$ and, in a $\Z$-equivariant way, populate every horizontal row with elements of $\bx_\bot^2$ so that no translation in $\Z^2$ fixes the output. 
\begin{proof}
    First fix any $\ax\in \Axn(\Gamma, \s C)$. Without loss of generality, we can identify $\ax$ with a subshift $X$ of $\s C^\Delta$. Let $R$ be a set of double coset representatives for $\Delta$ in $\Gamma$ so that $R\cap \Delta=\{\id\}$. This means that any $\gamma\in \Gamma$ can be written uniquely as $\gamma=\alpha r\beta$ with $r\in R$. Define $f: X\rightarrow \s C^{\Gamma\times \Delta}$ by 
    \[f(x)(\alpha r\beta,\delta)=x(\alpha\beta\delta).\] This lands in the correct space as, for any $\alpha,\beta,\eta,\delta\in \Delta$ and $r\in R$, \[f(x)(\alpha r \beta \eta,\delta)=x(\alpha\beta\eta\delta)=f(x)(\alpha r\beta, \eta\delta)\] and for any $x\in X$ \[f(x)(\alpha r\beta,\cdot)=x(\alpha\beta\cdot)=(\alpha\beta)\inv \cdot x\in X.\]

    Now we want to use $\bx_\bot^2$ to get a map into the free part of $\coind_\Delta^\Gamma(\ax\times \bx_\bot^2)\res \Delta$. It suffices to give a map in 
    \[\hom\bigl(\bx_\bot^2,F(\coind_{\Delta}^\Gamma(\bx_\bot^2))\res \Delta\bigr).\] 

    Let $\pi$ be an automorphism of $\bx_\bot^2$ so that $\pi(y)\not=y$ for any $y$. (Recall that $\bx_\bot$ a product of shifts of proper colorings, so we can build $\pi$ by permuting colors, for instance).

    For $(x,y)\in \bx_\bot^2$, define $f(x,y)=\bigl(f_0(x),f_1(y)\bigr)\in \coind_{\Delta}^\Gamma(\bx_\bot^2)$ by
    \[f_0(x)(\alpha r \beta,\delta)=x(\alpha\beta\delta)\] and \[f_1(y)(\alpha r\beta,\delta)=\left\{\begin{array} {cc} \pi(y)(\alpha\beta\delta)  & r\not=\id \\  y(\alpha\beta\delta) & r\not=\id\end{array} \right.\] for $\alpha,\beta,\delta\in \Delta$ and $r\in R$. 

    By the same algebra as before, this is indeed a map from $\bx_\bot^2$ into $\coind_\Delta^\Gamma(\bx_\bot^2)\res \Delta.$ We just have to check that $f$ lands in the free part of coinduced action. Suppose towards contradiction that $\gamma\cdot f(x,y)=f(x,y)$ for some $\gamma\not=\id.$ By free-ness of $\bx_\bot$, we must have $\gamma\not\in \Delta$, so this means that, if $\gamma\inv=\alpha r\beta$, then
    \[\bigl((\alpha\beta)\inv\cdot x,(\alpha\beta)\inv\cdot\pi(y)\bigr)=(x, y).\]
    There are two cases: if $\alpha\beta=\id$, then $\pi(y)=y$, but this is impossible by choice of $\pi$; if $\alpha\beta\not=\id$, then $(\alpha\beta)^2\cdot x=x$, but this is impossible since $x\in \bx_\bot$.
\end{proof}

\begin{prop}
    If $\Z$ is a subgroup of $\Gamma$, then $\wdth(\kfree(\Gamma))=|\R|$.
\end{prop}
\begin{proof}
    Let $A$ be the antichain of $\Z$-actions from Proposition \ref{prop: z antichain}. For $\ax\in A$, let $\ax'=\coind_\Z^\Gamma\bigl(\ax\times (\bx_{\bot}^\Z)^2\bigr)$. Fix $\ax,\bx\in A$, we want to show that $\ax'\not\preccurlyeq\bx'$.

    By the comments following Proposition \ref{prop: z antichain}, there is a shift of finite type $X$ so that ${\hom(\ax,X)\not=\emptyset}$ and $\hom(\bx, X)=\emptyset$. We will show that $X'=\coind_\Z^\Gamma(X)$ similarly serves as a witness to $\ax'\not\preccurlyeq\bx$. Note that $X'$ is a shift of finite type.
    
    The restriction $\ax'\res \Z$, factors onto $\ax$ by projection to the coordinates in $\Z$. In particular, by the adjoint property, $\hom(\ax', \coind_\Z^\Gamma(X))\not=\emptyset.$ 

    Now, suppose towards contradiction that $\bx'$ has a map to $X'$. Then, $\bx'\res \Z$ has a map to $X$. However, by the previous lemma, there is a map $f\in \hom\bigl(\bx \times (\bx_\bot^\Z)^2,\bx'\res\Z\bigr)$. But since $(\bx_\bot^\Z)^2\times\bx\approx \bx$ this means $\hom(\bx, X)\not=\emptyset$, which is a contradiction.
\end{proof}

Likely, $\wdth(\kfree(\Gamma))=|\R|$ for any infinite, finitely generated group. The proposition above reduces the question to studying monsters like Burnside groups. 

If we strip away finite noise by focusing on mixing actions, $\Z$ has only one free mixing action up to weak equivalence. More generally:

\begin{prop}
    If $\Gamma$ is 2-ended, then $\bx_{\bot}\approx \ax$ whenever $\ax$ is free and mixing.
\end{prop}
\begin{proof}
    The basic idea is that we simulate any labelling of $\ax$ on a free action $\bx$ by taking a well-spaced family of finite cuts, defining a continuous labelling however we like on these cuts in $\sch(\bx)$, then filling in the gaps using the fact that $\ax$ is mixing.

    Fix a generating set $E$ for $\Gamma$ (just so that we can speak unambiguously of components and ends in $\cay(\Gamma, E)$). Fix a continuous labelling $f:\s C\rightarrow A$ and window $W\Subset\Gamma$. And, fix a free action $\bx\in \kfree(\Gamma)$. We want to show that there is a continuous labelling $g: \s C\rightarrow A$ so that $p_W(f,\ax)=p_W(g,\ax).$
    
    Let $P=p_W(f,\ax)$, and let $X_P$ be the shift of finite type generated by $P$, i.e.
    \[X_P=\{x\in A^\Gamma: (\forall \gamma\in \Gamma)\; (\gamma\cdot x)\res W\in P\}.\] Note that, since $\ax$ is mixing, for any $F\Subset \Gamma$ and $\sigma,\tau\in p_F(f,\ax)$ and for all but finitely many $\gamma$, there is $x\in X_p$ so that $x\res W=\sigma$ and $(\gamma\cdot x)\res W=\tau$.

    Let $F\supseteq W$ be a finite subset of $\Gamma$ so that $\Gamma\setminus F$ has exactly $2$-components and so that if $\gamma, \delta$ are in different $\Gamma\setminus F$-components, then $\gamma\cdot W\cap \delta\cdot W=\emptyset.$ For each $\sigma\in P$, let $\sigma'\in A^F$ be some element of $p_F(\ax, f)$ extending $\sigma$. 

    Now we'll start building our labelling $g$. First, using free-ness of $\bx$ we can find a set $C=\bigcup_{\sigma\in P}C_\sigma$ so that 
    \begin{enumerate}
        \item The $C_\sigma$'s are clopen, disjoint, and nonempty
        \item $F\cdot x\cap F\cdot y=\emptyset$ for $x,y\in \bigcup_{\sigma\in P} C_\sigma$
        \item The connected components of $\s C\setminus F\cdot C$ as a subgraph of $\sch(\bx, E)$ are uniformly bounded in size. 
        \item Any path in $\sch(\bx, E)$ between distinct elements of $C_\sigma$ passes through $F\cdot C\setminus C_\sigma$
    \end{enumerate} Note that, if $K$ is a component of $\s C\setminus F\cdot C$, then $K$ is adjacent to exactly two cuts $F\cdot x$ and $F\cdot y$ with $x,y\in C$, and (by $(4)$) $x,y$ cannot both be in the same $C_\sigma$. 
    
    Now we'll start building our labelling of $\bx$. Fix any linear order $\sqsubseteq$ on $P$ and define $g$ as follows:
    \begin{enumerate}
        \item Suppose $y\in F\cdot C_\sigma$. Say $y=\gamma\cdot x$ for $x\in C$ and $\gamma\in F$. Then, set $g(y)=\sigma'(\gamma)$.
        \item Suppose $K$ is a component of $\s C\setminus F\cdot  C$. Say $K$ is adjacent to $F\cdot x, F_\cdot y$, $x\in C_\sigma$, $y\in C_\tau$, and $\sigma\sqsubseteq \tau$, suppose $y=\gamma\cdot x$ and let $z\in X_p$ be some labelling extending $\sigma\cup \gamma\cdot \tau$. For $\delta\cdot x\in K$, set $g(x)=z(\delta).$
    \end{enumerate}

    For any $x\in \s C$, $W\cdot x$ is either contained in $F\cdot y$ for some $y\in C$ or $W\cdot x$ is contained in a components of $\s C\setminus C$ and one of its two cuts. In either case, $g(\delta\cdot x)$ agrees with $z(\delta)$ for some $z\in X_p$. Thus, $p_W(g,\bx)\subseteq P$. And if $x\in C_\sigma$, then $g(\delta\cdot x)=\sigma(x)$, so $P\subseteq p_W(g,\bx).$
    
\end{proof}

This proof relied heavily on 2-ended-ness. Indeed, it seems likely (modulo some monstrosities) that the space of free mixing actions of $\Gamma$ is wide exactly when $\Gamma$ is not 2-ended. We can prove this for any group containing $F_2$ or $\Z^2$. 

First, we'll show that, for $\Gamma=\Z^2$ or $F_2$, $\wdth(\kfree\cap\kmix)=|\R|$. In fact, we have strong witnesses as above.

\begin{thm}
There is an antichain $A\subseteq \kfree(\Z^2)\cap\kmix(\Z^2)$ so that $|A|=|\R|$ and for any $\ax,\bx\in A$ there is a shift of finite type $X$ with $\hom(\ax,X)\not=\emptyset$ and $\hom(\bx,X)=\emptyset$.
\end{thm}
\begin{proof}
    For $p$ a prime, let $T_p$ be the subshift of tilings of $\Z^2$ with connected pieces of size $p$. For $S$ an infinite set of primes, let $\ax_S=\bx_\bot\times \prod_{p\in S} T_p.$ Since each $T_p$ is mixing, so is $\ax_S$. Since $\bx_\bot$ is free, so is $\ax_S.$ Of course, $\ax_S$ is an action on a perfect compact metric space, so we may identify $\ax_S$ with an action in $\s C$. For any $p\in S$, $\ax_S$ has a map into $T_p$. Our antichain is \[A=\{\ax_S: S\mbox{ an infinite set of primes}\}.\] All that remains to show is that $\ax_S$ has no map into $T_p$ when $p\not\in S$.

   Fix an ultrafilter $\s U$ on $\N$, and say $S=\{p_i:i\in\N\}$. Define $n_i=\prod_{i=1}^n p_i$, and let $\bx_i:\Z^2\curvearrowright \Z^2/n_i\Z^2$. Define
    \[\bx_S=\lim_{\s U} \bx_i.\] By freeness, $\bx_S$ maps to $\bx_\bot$. For $p\in S$, almost every $\bx_i$ maps to $T_p$. So, by \L o\'s's theorem (Theorem \ref{thm:los}) $\bx_S$ maps to $T_p$ when $p\in S$. Thus, $\hom(\bx_S, \ax_S)\not=\emptyset.$ 
    
    But, if $p\not\in S$, $n_i\times n_i$ is not divisible by $p$, so no $\bx_i$ has a map to $T_p$. Again by \L o\'s's theorem, $\hom(\bx_S, T_p)=\emptyset.$ Thus, when $p\not\in S$, $\hom(\ax_S, T_p)=\emptyset$ as desired.
\end{proof}

Almost the same proof gives an antichain in $F_2$, but tiling finite actions of $F_2$ is not so simple. We will use a combinatorial lemma about large girth graphs to cook up an appropriate sequence of actions and then modify the tiling problem slightly.

\begin{lem}\label{lemma:girth}
    For any $g\in \N$ and any large enough tree $T$ with all degrees at most $4$, there is a $4$-regular graph $G$ of girth at least $g$ obtained by adding edges to $T$. 
\end{lem}

In fact, if we add edges to $T$ at random between vertices which are at least distance $g$, we obtain a $4$-regular graph with high probability. See \cite{LinialGirth}.

\begin{lem}\label{lemma:action making}
    Every $2n$-regular graph is the Schreier graph of an action of $F_n$.
\end{lem}

The case of $4$-regular graphs is a tough exercise. For a general statement about $2d$-regular graphs see  \cite{SchreierGraphs}.

\begin{thm}
There is an antichain $A\subseteq \kfree(F_2)\cap\kmix(F_2)$ so that $|A|=|\R|$ and for any $\ax,\bx\in A$ there is a shift of finite type $X$ with $\hom(\ax,X)\not=\emptyset$ and $\hom(\bx,X)=\emptyset$.
\end{thm}

\begin{proof}
    As before, let $T_p$ be the subshift of tilings of $F_2$ by connected pieces of size $p$. And, for an infinite set of primes $S$, let \[\ax_S=\bx_\bot\times \prod_{p\in S} T_p,\quad A=\{\ax_S: S\mbox{ is an infinite set of primes}\}.\] Also as before, each $\ax_S$ is free, mixing, and conjugate to an action on $\s C$. 
    
    We need to verify that every pair of actions in $A$ is separated by a shift of finite type. Of course, if $p\in S$, $\hom(\ax_S, T_p)\not=\emptyset$. It remains to show that if $p\not\in S$ then $\hom(\ax_S, T_p)\not=\emptyset$. As before, we will construct an ultrapower which has a map to $\ax_S$, but no map to $T_p$.

    We need a sequence of finite graphs to play the role of the finite tori $\Z^2/p\Z^2$ from the previous argument. First, let $P_n$ be a path graph so that $p_1p_2...p_n$ divides $|P_n|$ and no other prime smaller than $p_n$ divides $|P_n|$, and so that $P_n$ is large enough to apply Lemma \ref{lemma:girth} with girth $n$. Let $G_n$ be the $4$-regular graph of girth at least $n$ extending $P_n$ guaranteed by the lemma.
    
    Since $G_n$ is $4$-regular, Lemma \ref{lemma:action making} says there is an action $\bx_n$ of $F_2$ with $\sch(\bx_n)=G_n$. Set $\bx=\colim_{\s U} \bx_n$. Since $\operatorname{girth}(G_n)\rightarrow\infty,$ $\bx$ is free. By saturation of ultraproducts, this means $\bx$ factors onto $\bx_\bot$. For $q\in S$, $P_n$ (and thus $G_n$) can eventually be tiled by pieces of size $q$; so $\bx$ maps into $T_q$. Thus $\hom(\bx, \ax_S)\not=\emptyset$. And, since $p\not\in S$, eventually $|G_n|$ is not divisible by $p$, so $\bx_n$ has no map into $T_p$. Thus $\hom(\bx, T_p)=\emptyset$. 

    It follows that $\hom(\ax_S, T_p)=\emptyset$ as desired.
    
\end{proof}

And, these antichains pass to supergroups.

\begin{cor}
    If $\Gamma$ contains $F_2$ or $\Z^2$ as a subgroup, then \[\wdth(\kmix(\Gamma)\cap\kfree(\Gamma))=|\R|.\]
\end{cor}
\begin{proof}
    Suppose $\Delta\subseteq \Gamma$ is either $F_2$ or $\Z^2$. Let $A$ be the antichain from the corresponding theorem above. For $\ax\in A$ let $\ax'=\coind_{\Delta}^\Gamma(\ax\times \bx_\bot)$, and set \[A'=\{\ax': \ax\in A\}.\] For $\ax, \bx\in A$, let $X$ be a shift of finite type witnessing that $\ax\not\preccurlyeq\ax$ and let $X'=\coind_{\Delta}^\Gamma(X)$. Then, by Lemma \ref{lemma:coinduction map}, $X'$ witnesses that $\ax'\not\preccurlyeq\bx'$. So $A'$ is antichain.
\end{proof}
\begin{cor}
    If $\Gamma$ has infinitely many ends, then $\wdth(\kmix(\Gamma)\cap\kfree(\Gamma))=|\R|$.
\end{cor}
\begin{proof}
    It suffices to show that any infinitely ended group has $F_2$ as a subgroup. This seems to be well-known folklore, but we'll include a sketch below.

    By Stalling's theorem, an infinitely ended group is either an amalgamated free product over a finite subgroup--$\Gamma\star_H\Delta$ with $[\Gamma:H]\geq 3$ and $[\Delta:H]\geq 2$--or an HNN extension over a finite subgroup--$\Gamma\star_\alpha=\ip{\Gamma, t: \{tht\inv\alpha(h): h\in H\}}$ for some $H<\Gamma$ with $[\Gamma:H]\geq 3$ and some isomorphic embedding $\alpha: H\rightarrow \Gamma$. 
    
    In the amalgamated free product case, we can find nontrivial coset representatives $\gamma_0,\gamma_1\in \Gamma$ and $\delta\in \Delta\setminus H$. One can check that $\gamma_0\delta$ and $\gamma_1\delta$ generate a free group.

    In the HNN extension case, we can pick $a,b\not\in H$ with $ab\inv\not\in H$. Then, $ta$ and $tb$ generate a free group.
\end{proof}

This argument can be extended a little to show that for groups like the lamplighter $\Z\wr(\Z/2\Z)$ the space of mixing actions is wide. But, for groups with no infinite finite quotients, a different method will be needed to build anti-chains in the weak-containment order (or rule them out).

\section{Open problems}

We'll end by collecting some open problems. Many of these are inspired by analogous problems and theorems for measure theoretic weak containment, and we will comment on these correspondences.

Hyperfiniteness plays a central role in the theory of measure theoretic weak containment. Hyperfiniteness is a measure weak equivalence invariant; all hyperfinite free actions of a group are measure weakly equivalent; and, a group is amenable if and only if it admits a free hyperfinite pmp action. 

It is unclear what the exact relation between topological weak equivalence and hyperfiniteness is, or if there is some topological version of hyperfiniteness which can play a similar role for topological actions.

\begin{prb}
    If $\ax$ is hyperfinite and $\ax\preccurlyeq\bx$, does it follow that $\bx$ is hyperfinite? 
\end{prb}
\begin{prb}
    Is the generic action hyperfinite for any amenable group? Any exact group? Any residually finite group?
\end{prb}
\begin{prb}
    Is almost-finiteness (in the sense of \cite{AlmostFinitness}) invariant under weak equivalence? 
\end{prb}

As mentioned in the discussion of $\kfin$, measure theoretic ultra-co-products of finite actions are essentially equivalent to so-called local-global limits of bounded degree graphs. Whether the Bernoulli shift $[0,1]^{F_2}$ is a local-global limit more or less amounts to determining if every problem that can be approximately solved on large girth degree $4$-graphs can be solved by a randomized local algorithm. We have seen that, for $\Gamma=\Z^2$, there is a list of problems $P$ so that every large enough finite torus can solve a problem in $P$, but there is no deterministic local algorithm to solve $P$. We might ask for a nice minimal list of such problems. More generally, we can ask for a basis for $\kfin$.

\begin{prb}
    Is it the case that, for every $\ax\in \kfin$, there is some $\bx\preccurlyeq\ax$ so that there is no $\mathbf c$ with $\mathbf c\preccurlyeq\bx$? For some $\Gamma$? For all $\Gamma$?
\end{prb}
\begin{prb}
    Is there a bottom element for $\kfin(F_2)$?
\end{prb}

In the measure theoretic setting, $\wdth(\kfree)$ detects amenability (it is $1$ iff $\Gamma$ is amenable). It is an open question if it is always $|\R|$ for nonamenable groups, and known to be $|\R|$ when $\Gamma$ contains a nonamenable free group. In the topological setting, it appears that $\wdth(\kfree)$ is always $|\R|$ when $\Gamma$ is infinite and finitely generated, and that $\wdth(\kfree\cap\kmix)$ picks out 2-endedness.

\begin{prb}
    Is there a finitely generated 1-ended group with $\wdth(\kfree\cap\kmix)<|\R|$?
\end{prb}
\begin{prb}
    Is there a finitely generated infinite group with $\wdth(\kfree)<|\R|$?
\end{prb}

Lastly, there are some phenomena which are genuinely unique to the topological setting. In the measure theoretic setting there is only one standard probability space, but there is a vast zoo of compact Hausdorff spaces.

\begin{prb}
    Which spaces admit a top $\Z$-action? a top $\Z^2$-action? A top $F_2$-action?
\end{prb}
\begin{prb}
    Which spaces admit a bottom free $\Z$-action? A bottom free $\Z^2$-action? A bottom free $F_2$-action?
\end{prb}

More generally, what happens in higher dimensional spaces? Proposition \ref{prop:nonfree limit} suggests the answer may have some relation to Bourgin--Yang theory.

\bibliographystyle{plain}
\bibliography{refs}

\begin{thebibliography}{10}

\bibitem{ContinuousCombo}
Anton Bernshteyn.
\newblock Probabilistic constructions in continuous combinatorics and a bridge to distributed algorithms.
\newblock {\em Advances in Mathematics}, 415:108895, 2023.

\bibitem{BurtonKechrisSurvey}
Peter~J. Burton and Alexander~S. Kechris.
\newblock Weak containment of measure-preserving group actions.
\newblock {\em Ergodic Theory and Dynamical Systems}, 40(10):2681–2733, 2020.

\bibitem{GaoCamerlo}
Riccardo Camerlo and Su~Gao.
\newblock The completeness of the isomorphism relation for countable boolean algebras.
\newblock {\em Transactions of the AMS}, 2000.

\bibitem{BAD}
Clinton Conley, Steve Jackson, Andrew Marks, Brandon Seward, and Robin Tucker-Drob.
\newblock Borel asymptotic dimension and hyperfinite equivalence relations.
\newblock {\em Duke Mathematical Journal}, 172:3175–3226, 2023.

\bibitem{DouchaGeneric}
Michal Doucha.
\newblock Strong topological rokhlin property, shadowing, and symbolic dynamics of countable groups.
\newblock {\em Journal of the European Mathematical Society}, 2025.

\bibitem{ElekQW}
Gabor Elek.
\newblock Qualitativegraph limit theory. cantor dynamical systems and constant-time distributed algorithms.
\newblock {\em preprint. ArXiv: 1812.07511}, 2019.

\bibitem{Engelking}
R.~Engelking.
\newblock {\em Dimension Theory}.
\newblock Mathematical Studies. North-Holland Publishing Company, 1978.

\bibitem{GelfandDuality}
G.B. Folland.
\newblock {\em A Course in Abstract Harmonic Analysis}.
\newblock Studies in Advanced Mathematics. Taylor \& Francis, 1994.

\bibitem{GaoIDST}
Su~Gao.
\newblock {\em Invariant Descriptive Set Theory}.
\newblock Chapman \& Hall/CRC Pure and Applied Mathematics. CRC Press, 2008.

\bibitem{SchreierGraphs}
Jonathan~L Gross.
\newblock Every connected regular graph of even degree is a schreier coset graph.
\newblock {\em Journal of Combinatorial Theory, Series B}, 22(3):227--232, 1977.

\bibitem{HarringtonMarkerShelah}
Leo Harrington, David Marker, and Saharon Shelah.
\newblock Borel orderings.
\newblock {\em Trans. Amer. Math. Soc.}, 310(1):293--302, 1988.

\bibitem{hjorth}
G.~Hjorth.
\newblock {\em Classification and Orbit Equivalence Relations}.
\newblock Mathematical surveys and monographs. American Mathematical Society, 2000.

\bibitem{hochman}
Michael Hochman.
\newblock Genericity in topological dynamics.
\newblock {\em Ergodic Theory and Dynamical Systems}, 28(1):125–165, 2008.

\bibitem{GlobalAspects}
Alexander Kechris.
\newblock {\em Global Aspects of Ergodic Group Actions}, volume 160.
\newblock AMS Mathematical Surveys and Monographs, 2010.

\bibitem{AlmostFinitness}
David Kerr.
\newblock Dimension, comparison, and almost finiteness.
\newblock {\em Journal of the European Mathematical Society}, 11:3697--3745, 2020.

\bibitem{KrawczykSteprans}
Adam Krawczyk and J.~Steprāns.
\newblock Continuous colourings of closed graphs.
\newblock {\em Topology and its Applications}, 51(1):13--26, 1993.

\bibitem{LinialGirth}
Nati Linial and Michael Simkin.
\newblock A randomized construction of high girth regular graphs.
\newblock {\em Random Structures and Algorithms}, 2020.

\bibitem{PercolationWeakContainment}
RUSSELL LYONS.
\newblock Fixed price of groups and percolation.
\newblock {\em Ergodic Theory and Dynamical Systems}, 33(1):183–185, 2013.

\bibitem{Marker}
D.~Marker.
\newblock {\em Model Theory : An Introduction}.
\newblock Graduate Texts in Mathematics. Springer New York, 2002.

\bibitem{BorsukUlam}
Jiri Matousek.
\newblock {\em Using the Borsuk-Ulam Theorem: Lectures on Topological Methods in Combinatorics and Geometry}.
\newblock Universitext (Berlin. Print). Springer, 2003.

\bibitem{IyerShinko}
Forte Shinko and Sumun Iyer.
\newblock Asymptotic dimension and hyperfiniteness of generic cantor actions.
\newblock {\em Groups, Geometry, and Dynamics}, 2025.

\bibitem{HypergraphWeakContainment}
Riley Thornton.
\newblock Limits of sparse hypergraphs.
\newblock {\em preprint. ArXiv: 2410.17483}, 2025.

\bibitem{TodorTomas}
Todor Tsankov and Tomas Ibarlucia.
\newblock A model-theoretic approach to rigity of strongly ergodic, distal actions.
\newblock {\em Ann. Sci. Ec. Norm. Super.}, 4:751--780, 2021.

\bibitem{Survey}
Itaï~Ben Yaacov, Alexander Berenstein, C.~Ward Henson, and Alexander Usvyatsov.
\newblock {\em Model theory for metric structures}, page 315–427.
\newblock London Mathematical Society Lecture Note Series. Cambridge University Press, 2008.

\bibitem{AndyTopWeak}
Andrew Zucker.
\newblock Ultracoproducts and weak containment for flows of topological groups.
\newblock {\em preprint. ArXiv:2401.08000}, 2024.

\end{thebibliography}

\end{document}